\documentclass{amsart}
\usepackage{amssymb, amsmath}
\usepackage{stmaryrd}
\usepackage[all]{xy}
\usepackage{mathrsfs}
\usepackage{todonotes}
\usepackage{tikz-cd}

\newcommand{\cyc}{\textrm{cyc}}
\newcommand{\Ch}{\textrm{Ch}}
\newcommand{\m}[1]{\underline{#1}}
\newcommand{\mR}{\m{R}}
\newcommand{\mM}{\m{M}}
\newcommand{\mN}{\m{N}}
\newcommand{\mA}{\m{A}}
\newcommand{\mP}{\m{P}}
\newcommand{\mZ}{\m{\Z}}
\newcommand{\gmM}{{}^{g}\!\mM}
\newcommand{\mMg}{\mM^{g}}  
\newcommand{\Z}{\mathbb Z}
\newcommand{\bZ}{\mathbb Z}
\newcommand{\N}{\mathbb N}

\newcommand{\bS}{\mathbb S}

\newcommand{\Green}{\mathcal Green}
\DeclareMathOperator{\Sym}{Sym}

\newcommand{\tildeEF}{\tilde{E}\mathcal F}
\newcommand{\EF}{E\mathcal F}
\DeclareMathOperator{\sd}{sd}
\DeclareMathOperator{\THH}{THH}
\DeclareMathOperator{\HH}{HH}
\newcommand{\cA}{\mathcal A}

\newcommand{\cC}{\mathcal C}
\DeclareMathOperator{\Hom}{Hom}
\DeclareMathOperator{\End}{End}
\DeclareMathOperator{\Lan}{Lan}
\newcommand{\cO}{\mathcal O}
\newcommand{\Sp}{\mathcal Sp}
\DeclareMathOperator{\boxtor}{Tor}
\DeclareMathOperator{\CoInd}{CoInd}
\DeclareMathOperator*{\holim}{holim}
\mathchardef\mhyphen=45

\newcommand{\op}{\textrm{op}}
\newcommand{\Map}{\textrm{Map}}
\newcommand{\Mack}{\textrm{Mack}_G}
\newcommand{\sMack}{\textrm{sMack}_G}

\newcommand{\Set}{\mathcal Set}
\newcommand{\htp}{\simeq}
\newcommand{\Tor}{\textrm{Tor}}
\newcommand{\Ab}{\mathcal Ab}
\newcommand{\Tamb}{\mathcal Tamb}
\newcommand{\cOTamb}{\cO\mhyphen\Tamb}

\newcommand{\tr}{tr}
\newcommand{\res}{res}

\newcommand{\HC}[1]{\m{HC}^{#1}}

\newcommand{\sm}[1]{\ensuremath{\mathop{\wedge}_{#1}}}
\newcommand{\sma}{\wedge}

\newcommand{\boxover}[1]{\ensuremath{\mathop{\Box}_{#1}}}

\numberwithin{equation}{section}

\newtheorem{theorem}[equation]{Theorem}
\newtheorem*{theorem*}{Theorem}
\newtheorem{corollary}[equation]{Corollary}
\newtheorem{lemma}[equation]{Lemma}
\newtheorem{proposition}[equation]{Proposition}

\theoremstyle{definition}
\newtheorem{definition}[equation]{Definition}
\newtheorem{remark}[equation]{Remark}
\newtheorem{example}[equation]{Example}

\begin{document}

\title{The Witt vectors for Green functors}

\author[A.J. Blumberg]{Andrew~J. Blumberg}
\address{University of Texas, Austin, TX 78712}
\email{blumberg@math.utexas.edu}

\author[T.Gerhardt]{Teena Gerhardt}
\address{Michigan State University, East Lansing, MI 48824}
\email{teena@math.msu.edu}

\author[M.A.Hill]{Michael~A. Hill}
\address{University of California Los Angeles\\Los Angeles, CA 90025}
\email{mikehill@math.ucla.edu}

\author[T.Lawson]{Tyler Lawson}
\address{University of Minnesota, Minneapolis, MN 55455}
\email{tlawson@math.umn.edu}

\keywords{Mackey functors, Green functors, Hochschild homology, Witt vectors, topological Hochschild homology}

\begin{abstract}
We define twisted Hochschild homology for Green functors.  This
construction is the algebraic analogue of the relative topological
Hochschild homology $\THH_{C_n}(-)$, and under flatness conditions it describes the $E_2$ term of the
K\"unneth spectral sequence for relative $\THH$.  Applied to ordinary
rings, we obtain new algebraic invariants.  Extending Hesselholt's
construction of the Witt vectors of noncommutative rings, we interpret
our construction as providing Witt vectors for Green functors.
\end{abstract}

\maketitle

\section{Introduction}

The definition of topological Hochschild homology ($\THH$) is one of
the most successful examples of the program of ``brave new algebra'',
in which classical algebraic constructions are carried out ``over the
sphere spectrum''.  In the case of $\THH$, one literally replaces rings
with ring spectra and tensor products over $\bZ$ with smash products
over $\bS$ to pass from Hochschild homology to topological Hochschild
homology.

One byproduct of the foundational work of Hill-Hopkins-Ravenel has
been a new interpretation of $\THH$ as the $S^1$-norm $N_e^{S^1}$ from
the trivial group to $S^1$~\cite{normTHH,Stolz}.  This description
immediately suggests that one might consider relative versions of
$\THH$ which take as input $H$-equivariant ring spectra for $H
\subseteq S^1$; the $H$-relative $\THH_H$ is then defined to be the
norm $N_H^{S^1}$~\cite[8.2]{normTHH}.  We can also produce an
$H$-relative theory of topological cyclic homology (as well as $H$-relative versions of $TF$ and $TR$). 

At this point, a natural question arises: what is the algebraic
analogue of this equivariant relative $\THH$?  The goal of this paper
is to answer this question: we define and study twisted Hochschild
homology for Green functors.  To explain the nature of our
construction, it is helpful to recall the role of Mackey functors and
Green functors in equivariant homotopy theory.

Mackey functors play the role that abelian groups play in
non-equivariant homotopy theory.  In particular, an equivariant spectrum has
homotopy Mackey functors, rather than just homotopy groups.  The
category of Mackey functors can be endowed with a symmetric monoidal
structure via the box product, and we can talk about associative or
commutative monoids in this category.  These are associative or
commutative Green functors.  The category of Mackey functors (and more
generally, the category of $\m{R}$-modules for a Green functor
$\m{R}$) has enough projectives, and if we consider only finite
groups, then these are all flat.  One could therefore attempt to
directly mimic the classical definition of Hochschild homology,
forming a derived tensor product of a Green functor $\m{R}$ with
itself over its enveloping algebra.

Unfortunately, this version of Hochschild homology is not the one that
we want, for several reasons. From a philosophical standpoint, it is
too general --- this definition makes sense for any finite group $G$,
whereas our equivariant $\THH$ can only be defined for finite
subgroups of $S^{1}$.  Second, there is no obvious relationship
between the $K$-relative Hochschild homology and the $H$-relative
Hochschild homology for $K\subset H$ finite subgroups of $S^{1}$, in
contrast to the situation with $N_H^{S^1}$ and $N_K^{S^1}$.  Finally
and perhaps most importantly, the naive Hochschild homology does not
receive an obvious edge map from the topological version.

For $H \subset G$, the restriction functor from $G$-spectra to $H$-spectra, given by restricting the action to $H$, is denoted $i_H^*$.  Let $K\subset H\subset
S^{1}$ be finite subgroups, and let $R$ be an associative ring orthogonal $K$-spectrum. There is an extension of the K\"unneth spectral sequence from ~\cite{normTHH}
which can be used to compute the homotopy Mackey functors of the relative topological Hochschild homology
$i_{H}^{\ast} THH_{K}(R)$.  Specifically, when Adams indexed, this spectral sequence has the
form
\[
E_{2}^{s,t}=\m{\Tor}^{\m{\pi}_{\ast}
  N_{K}^{H}R^{e}}_{-s}\Big(\m{\pi}_{\ast} (N_{K}^{H}R),\Sigma^{t}\big({}^{\gamma}
\m{\pi}_{\ast} (N_{K}^{H}R)\big)\Big)\Rightarrow
\m{\pi}_{t-s} i_{H}^{\ast} THH_{K}(R),
\]
where the superscript $\gamma$ indicates a particular twist of the
bimodule structure and $\m{\Tor}$ denotes the derived functors of the
symmetric monoidal product.  In this K\"unneth spectral sequence, we
see that we are not considering the $\m{\Tor}$ Mackey functors of a
fixed module with itself: instead, we are $\m{\Tor}$-ing together a
module with a particular twist of itself.

These observations guide our definition of a twisted Hochschild
homology of Green functors for finite subgroups of $S^{1}$.  Our
definition is purely algebraic, building a differential graded Mackey
functor out of a Green functor $\m{R}$ (and building a dg Green
functor if $\m{R}$ is commutative).

For $G \subset S^1$ a finite subgroup, and $\mR$ an associative Green functor for $G$, we define the {\em twisted
  Hochschild complex} as a simplicial Mackey functor for $G$ with
$k$-simplices given by
\[
\HC{G}_{k}(\mR):= \mR^{\Box (k+1)}.
\]
The structure maps are the usual Hochschild maps except that the
$k$\textsuperscript{th} face map $d_{k}$ is given by moving the last factor to the front, acting via the
generator $g = e^{2\pi i / |G|} \in S^1$ of $G$ on the new first factor, and then multiplying the first two factors.

The {\em $G$-twisted Hochschild homology} of $\mR$, denoted
$\m{\HH}_i^G(\m{R})$, is the homology of this complex. This is first defined and studied in
Section~\ref{ssec:TwistedHH}.

For the comparison with topological Hochschild homology, we will also need a relative version of this theory. The construction of relative Hochschild homology for Green functors relies on Mackey functor norms. Let $K \subset G \subset S^{1}$ be finite subgroups, and let $\mR$ be
an associative Green functor for $K$. We will define the {\em twisted
  Hochschild complex relative to $K$} as a simplicial Mackey functor for $G$ with
$k$-simplices given by
\[
\HC{G}_{K}(\mR)_k:= \big(N_{K}^{G}\mR\big)^{\Box (k+1)}.
\]
Again, the structure maps are the usual Hochschild maps except that the
$k$\textsuperscript{th} face map moves the last factor to the front, acts via the
generator $g = e^{2\pi i / |G|} \in S^1$ of $G$ on the new first factor, and then multiplies the first two factors. The {\em $G$-twisted Hochschild homology} of $\mR$ relative to $K$, denoted
$\m{\HH}_K^G(\m{R})_*$, is the homology of this complex. These are
compatible as $G$ varies. In Section \ref{sec:Relative} we study this relative twisted Hochschild homology.

The twisted Hochschild homology of $\m{R}$ relative to $K$ serves as an
algebraic approximation to the $K$-relative $\THH$ of the
Eilenberg-MacLane spectrum $H\m{R}$.  Recall that in the classical setting, for a ring $A$, the topological Hochschild homology of the Eilenberg-MacLane spectrum $HA$ and the Hochschild homology of $A$ are related by a linearization map
\[
\pi_k\THH(HA) \to \HH_k(A)
\]
that factors the Dennis trace from algebraic $K$-theory to Hochschild homology
\[
K_k(A) \to \pi_k\THH(HA) \to \HH_k(A).
\]
In Section \ref{sec:THH} we prove the following.

\begin{theorem}\label{thm:THHEdgeMap1}
For $H\subset G \subset S^1$ finite subgroups, and $R$ a $(-1)$-connected $H$-equivariant commutative ring spectrum, we have a natural homomorphism
\[
\underline{\pi}^G_{k} \THH_{H}(R)\to \m{\HH}_{H}^{G}(\m{\pi}^{H}_0R)_k,
\]
refining the classical maps.
\end{theorem}
Since the construction of twisted Hochschild homology is purely
algebraic, it also gives us new invariants of
associative and commutative rings. For $A$ a commutative ring, the linearization map of \ref{thm:THHEdgeMap1}  yields lifts of the Dennis trace to $G$-twisted Hochschild homology for each finite $G\subset S^1$
\[
K_{k}(A)\to \m{\HH}_{e}^{G}(A)_k(G/G).
\]

The twisted Hochschild complex of a Green functor has a kind
of cyclotomic structure.  We give a direct definition of the geometric
fixed points $\Phi^N$ for a simplicial Green functor (see
Definition~\ref{def:GFP}), and using this we prove the following
structural result about $\HC{G}$.

\begin{proposition}
Let $H \subset G \subset S^1$ be finite subgroups, and let $N$ be a subgroup of $G$. For $\mR$ a commutative Green functor for $H$, there is a natural isomorphism of simplicial Green functors
\[
\Phi^{N}\big(\HC{G}_{H}(\mR)_{\bullet}\big)\cong
\HC{G/N}_{HN/N}(\Phi^{H\cap N}\mR)_{\bullet}.
\]
\end{proposition}
Using this cyclotomic structure, in Section \ref{sec:Example} we define a purely algebraic analogue of TR$(A;p)$, and compute this algebraic analogue in the case of $A=\mathbb{F}_p$. 

A particularly interesting facet of our definition of $H$-twisted
Hochschild homology is that it can be interpreted as defining the {\em
  Witt vectors of a Green functor} (See Section~\ref{sec:GreenWitt}).  Work of Hesselholt and of
Hesselholt-Madsen shows that for a commutative ring $R$,
$$\pi_0(\THH(R)^{C_{p^{n}}}) \cong W_{n+1}(R),
$$
where $W_{n+1}(R)$ denotes the length $n+1$ $p$-typical Witt vectors of $R$~\cite{HesselholtMadsen}. For $H$-relative topological Hochschild homology, we prove that $\m{\pi}_0$ is captured by twisted Hochschild homology.

\begin{theorem}\label{thm:ConnectiontoTHH}
Let $H\subset G\subset S^1$ be finite subgroups and let $R$ be a
$(-1)$-connected commutative ring orthogonal $H$-spectrum. We have a
natural isomorphism
\[
\m{\pi}_0^G\big(\THH_H\!(R)\big)\cong \m{\HH}_H^G(\m{\pi}_0^H R)_0.
\]
\end{theorem}
This suggests the following definition:

\begin{definition}
Let $H\subset G\subset S^1$ be finite subgroups and let $\mR$ be a Green functor for $H$. The $G$-Witt
vectors for $\mR$ are defined by
\[
\m{\mathbb
  W}_{G}(\mR):=\m{\HH}_H^{G}(\mR)_0.
\]
\end{definition}
With this definition we have the following analog of the Hesselholt-Madsen result 
\begin{corollary}
For $H\subset G\subset S^1$ finite subgroups, and $R$ a (-1)-connected commutative ring orthogonal $H$-spectrum,
\[
\m{\pi}_0^G\THH_H(R) \cong \m{\mathbb{W}}_G(\m{\pi}_0^HR).
\]
\end{corollary}

One of the key insights in the theory of equivariant homotopical
algebra is that Green functors are not the correct algebraic analogues
of $E_\infty$ rings.  Specifically, the homotopy Mackey functors of an
equivariant commutative ring spectrum have more structure than simply
a Green functor: they are Tambara functors~\cite{Brun}.  These are
Green functors with multiplicative transfer maps called norms.

When working with Tambara functors, our equivariant Witt vectors
inherit additional structure.  When $\m{R}$ is a Tambara functor, then
the norm maps provide Teichm\"uller style lifts in the $C_{mn}$-Witt
vectors. We study this in Section~\ref{ssec:Teich}, proving the following theorem there.
\begin{theorem}
Let $\m{R}$ be a \(G\)-Tambara functor, where \(G=C_{n}\) is a finite cyclic group.  For all $H\subset G$, we have a
Teichm\"uller-style, multiplicative map
\[
\m{R}(G/H)\to \big(\m{\mathbb W}_{G}(\m{R})\big)(G/G).
\]
\end{theorem}
When we apply this to an ordinary, non-equivariant commutative ring, then this provides an algebraic explanation of the Teichm\"uller lifts in terms of constructions of equivariant stable homotopy theory.

In Section~\ref{sec:CyclicNerves}, we apply our twisted Hochschild homology to compute the
homotopy Mackey functors of the relative $\THH$ for a pointed monoid
ring.  Hesselholt and Madsen observed that if $M$ is a pointed monoid,
then we have a natural $S^1$-equivariant equivalence
\[
\THH(R[M])\simeq \THH(R)\wedge N^{\cyc}(M).
\]
In our $C_n$-relative context, there is a generalization of this
statement that yields an analogous decomposition for our twisted
Hochschild homology. Here we also allow \(C_{n}\) to act on the pointed monoid.

\begin{theorem}
Let $G=C_n$. If $M$ is a pointed monoid in $G$-sets and if $\m{R}$ is a
$G$-Green functor, then we have a natural weak equivalence
\[
\HC{G}_\ast(\m{R}[M])\simeq \HC{G}_\ast(\m{R})\Box
\m{C}^{cell}_{\ast}(N^{\cyc}_{G} M;\mA),
\]
where \(\mA\) is the Burnside Mackey functor and where \(\m{C}_{\ast}^{cell}\) is the Mackey extension of the Bredon cellular chains.
\end{theorem}

We believe there is a natural relationship between the construction of
twisted Hochschild homology of Green functors we study in this paper
and the definitions of Hochschild-Witt vectors introduced by
Kaledin~\cite{KaledinHW}.  Kaledin has introduced a
functorial Hochschild-Witt complex associated to an associative unital
$k$-algebra $A$.  When $A$ is commutative, finite type (as an algebra)
over $k$, and smooth, a version of the HKR theorem implies that this
construction recovers the de Rham-Witt complex of $A$~\cite[\S
  6]{KaledinHW}.  When $A$ is noncommutative, the zeroth homology
coincides with Hesselholt's Witt vectors~\cite{HesselholtWitt}.  Most
interestingly, the higher homology groups of this Hochschild-Witt
complex provides a (conjectural) algebraic model of $TR$.  We intend
to return to give a more complete treatment in future work.

Our treatment here focuses on the commutative case, and we largely restrict attention to commutative Green functors for our constructions. One could instead ask for the purely associative case, allowing us to build most easily a version of Shukla homology for associative Green functors. We sketch how this works in Section~\ref{sec:Shukla}, showing that the homotopy invariant versions of our constructions are, in fact, homotopy invariant. The main problem here is that the category of Mackey functors has infinite projective dimension \cite{GreenleesProjective}, and in practice, even the most basic non-projective Mackey functors have infinite projective dimension. This makes the construction of resolutions much trickier than in the classical case, and we leave this section as mainly a sketch.

\subsection*{Acknowledgments.}  The authors would like to thank John
Greenlees and Mike Mandell for helpful conversations. The authors also thank an anonymous referee for comments which helped improve the exposition of this paper. This work was supported by the National Science Foundation [DMS-1151577 to Blumberg, DMS-1149408 to Gerhardt, DMS-1509652 to Hill, and DMS-1610408 to Lawson]. This project
was made possible by the hospitality of the Hausdorff Research
Institute for Mathematics at the University of Bonn and the Mathematical Sciences Research Institute (MSRI).

\section{Homological constructions in Mackey functors}\label{sec:HomConst}

In this section, we construct the twisted Hochschild complex in Mackey functors. We begin by setting up necessary foundations. Throughout we fix a
finite group $G$.  

\begin{definition}
Let \(\cA_{\N}\) denote the Burnside category, the category with objects finite \(G\)-sets and with morphisms isomorphism classes of spans:
\[
\cA_{\N}(S,T)=\Big\{[S\leftarrow U\rightarrow T]\Big\},
\]
where two spans \(S\leftarrow U\to T\) and \(S\leftarrow U'\to T\) are isomorphic if there is an isomorphism \(U\to U'\) making the obvious triangles with \(S\) and \(T\) commute. Composition is via pullback.
\end{definition}

Since a span is also just a map \(U\to S\times T\), the Burnside category is self-dual. In particular, the disjoint union of finite \(G\)-sets is both the product and coproduct in the Burnside category. 
\begin{definition}
A semi-Mackey functor is a product preserving functor
\[
\mM\colon \cA_{\N}\to\Set.
\]
The category of semi-Mackey functors is the obvious category of functors and natural transformations.
\end{definition}

Since the Burnside category is pre-additive in the sense that finite coproducts and products exist and agree, the hom objects are all commutative monoids. This is a consequence of the pre-additivity, and it implies that for every finite set \(T\) and for every semi-Mackey functor \(\mM\), the set \(\mM(T)\) inherits the structure of a commutative monoid.

\begin{definition}
A Mackey functor is a semi-Mackey functor \(\mM\) such that for each finite \(G\)-set \(T\), the commutative monoid \(\mM(T)\) is an abelian group. Let \(\Mack\) denote the full subcategory of semi-Mackey functors spanned by the Mackey functors.
\end{definition}
Again, we stress that being a group is a property of a commutative
monoid, and the commutative monoid structure arises by
functoriality.  We will find it helpful in what follows to regard
Mackey functors as functors from \(\cA_{\N}\) to the 
category of sets; in particular, the categorical constructions we 
use are defined in terms of the underlying functors to sets.

For homological algebra reasons, it is helpful also to build several other equivalent forms of Mackey functors. 
\begin{definition}
Let \(\cA\) denote the category enriched in abelian groups for which the objects are finite \(G\)-sets and for which the hom objects \(\cA(S,T)\) are the group completions of \(\cA_{\N}(S,T)\).
\end{definition}
In the category \(\cA\), we again have that disjoint union is both the product and the coproduct, but now the category is additive. The inverses on the hom objects allow us to record the inverses in a Mackey functor, as opposed to a semi-Mackey functor.
\begin{proposition}
A Mackey functor is an additive functor 
\[
\mM\colon\cA\to\Ab.
\]
\end{proposition}
This shows that the category of Mackey functors is an abelian category.  Since $\cA$ is self-dual, we can equivalently talk about ``co-Mackey functors''.  

Note that any finite $G$-set can be written as a disjoint union of orbits $G/H$. Since a Mackey functor is additive, the value of a Mackey functor $\m{M}$ on any finite $G$-set is determined by its values on these orbits. 

\begin{remark}
For a Mackey functor $\m{M}$, the Weyl group $W_G(H) = N_G(H)/H$ acts on $\m{M}(G/H)$. These Weyl group actions will play a role in our definition of twisted Hochschild homology in Section \ref{ssec:TwistedHH}
\end{remark}

The following example is particularly relevant for this work.

\begin{example}
For a $G$-spectrum $X$, the equivariant homotopy groups of $X$ form Mackey functors $\m{\pi}_n^G$:
$$
\m{\pi}_n^G(X)(G/H) := \pi_n(X^H). 
$$
\end{example}

We note that when discussing the model structure, we will also find it useful to describe Mackey functors as modules over the ``Mackey algebra'' $\mu_G(\bZ)$~\cite[\S3]{ThevanazWebb}.  

We next describe the symmetric monoidal structure on $\Mack$.

\subsection{The box product of Mackey functors}\label{ssec:BoxProduct}

We now review the closed symmetric monoidal structure on the category
of Mackey functors (see also~\cite[\S 1]{LewisMandell} for a review of
this structure).  The box product of Mackey functors $\mM \Box \m{N}$ is a Day convolution
product, defined by left Kan extension over the Cartesian product of finite $G$-sets:
\[
\xymatrix{ {\cA\times\cA}\ar[d]_{\mhyphen\times\mhyphen} \ar[r]^{\mM\times \m{N}} &
  {\mathcal Ab\times\mathcal Ab}\ar[r]^{\mhyphen\otimes\mhyphen} & {\mathcal Ab}
  \\ {\cA}\ar[urr]_{\mM\Box\m{N}} }
\]

\begin{theorem}
The box product of two Mackey functors is again a Mackey functor, and the category $\Mack$ is a symmetric monoidal category with product $\Box$ and unit $\mA=\cA(\ast,-)$.
\end{theorem}

This Kan extension takes place in additive categories, and this makes comparisons with later Kan extensions more difficult. We can, however, work instead entirely in \(\Set\)-valued functors, noting that a product preserving functor from the Burnside category to \(\Set\) is necessarily commutative monoid valued and that being an abelian group is a property of a commutative monoid. Work of Strickland makes this precise. 

\begin{proposition}[{\cite[Appendix A.3]{StricklandTambara}}]
The box product of two Mackey functors \(\mM\) and \(\m{N}\) is given by the left Kan extension of the Cartesian product of sets over the Cartesian product of finite \(G\)-sets:
\[
\xymatrix{
{\cA_{\N}\times\cA_{\N}}\ar[d]_{\mhyphen\times\mhyphen} \ar[r]^{\mM\times\m{N}} &
{\Set\times\Set}\ar[r]^{\mhyphen\times\mhyphen} & {\Set} \\
{\cA_{\N}}\ar[urr]_{\mM\Box\m{N}}
}
\]
\end{proposition}
In particular, all of the abelian group structure which we want on our putative symmetric monoidal product comes along for free on the ordinary left Kan extension.

To get a more explicit description, it helps to unpack this a little in terms of representable Mackey functors.

\begin{definition}\label{defn:Representables}
If $T$ is a finite $G$-set, then let
\[
\mA_T=\cA(T,-)
\]
be the functor represented by $T$.
\end{definition}

These are projective objects in the category of Mackey functors, and
by the Yoneda Lemma, we have a canonical isomorphism
\[
\Hom(\mA_T,\mM)\cong \mM(T)
\]
for any Mackey functor $\mM$.  Moreover, by the contravariant Yoneda
Lemma, the assignment
\[
T\mapsto \mA_T
\]
defines a co-Mackey functor object in the category of Mackey functors.
This structure will be applied below to determine the internal $\Hom$
in the category of Mackey functors.

Since the box product is defined by a Kan extension, the value on
representable functors is canonically determined by the following
formula:

\begin{proposition}\label{prop:boxrep}
If $T_1$ and $T_2$ are finite $G$-sets, then we have a canonical
isomorphism
\[
\mA_{T_1}\Box\mA_{T_2}\cong \mA_{T_1\times T_2}.
\]
\end{proposition}

Since the box product commutes with colimits in each variable by
construction, Proposition~\ref{prop:boxrep} allows us to reduce box
product computations to colimits of representables.

\begin{definition}
If $\mM$ is a Mackey functor and $T$ is a finite $G$-set, then let
\[
\mM_T:= \mA_T\Box \mM.
\]
\end{definition}

Generalizing the observation above, the $\mM_T$ fit together as a
co-Mackey functor:

\begin{proposition}\label{prop:comack}
For any Mackey functor $\mM$, the assignment
\[
T\mapsto \mM_T
\]
is a co-Mackey functor object in Mackey functors.
\end{proposition}

Proposition~\ref{prop:comack} allows us to describe an internal hom
object for $\Mack$.

\begin{definition}
If $\mM$ and $\m{N}$ are Mackey functors, then let
\[
\m{\Hom}(\mM,\m{N})(T):= \Hom(\mM_T,\m{N}).
\]
\end{definition}

It is straightforward to verify that $\m{\Hom}$ satisfies the expected
adjunctions with the box product.

\begin{theorem}
The category $\Mack$ is a closed symmetric monoidal category with
$\m{\Hom}$ as the internal mapping object.
\end{theorem}

Throughout this paper, we will see that constructions in Mackey functors are often closely related to analogous constructions in orthogonal $G$-spectra. The category $\Sp_G$ of orthogonal $G$-spectra for a complete universe
$U$ is a symmetric monoidal category under the smash product, and
this is compatible with the symmetric monoidal structure on Mackey
functors; the following compatibility result follows
from~\cite[1.3]{LewisMandell}.

\begin{proposition}\label{prop:pizerosym}
If $E$ and $F$ are cofibrant $(-1)$-connected genuine $G$-spectra, then we have
a natural isomorphism
\[
\m{\pi}_0 E\Box\m{\pi}_0 F\cong \m{\pi}_0 (E\wedge F).
\]
\end{proposition}

\begin{remark}\label{rem:Symm}
This isomorphism can be used to define the symmetric monoidal
structure on Mackey functors, using that the category of Mackey
functors is the heart of the category of genuine $G$-spectra:
\[
\mM\Box\mN\cong \m{\pi}_0 (H\mM\wedge H\mN).
\]
Here $H$ denotes an Eilenberg-Mac Lane functor
\[
H \colon \Mack \to \Sp_G
\]
that takes a Mackey functor to (a cofibrant version of) the associated Eilenberg-Mac Lane $G$-spectrum.
\end{remark}

Green functors are defined using the symmetric monoidal structure on $\Mack$. 

\begin{definition}
A (commutative) \emph{Green functor} for $G$ is an associative (commutative) monoid in the symmetric monoidal category $\Mack$.
\end{definition}

\subsection{The twisted cyclic nerve}\label{ssec:TwistedNcyc}

We now construct a cyclic nerve, which takes as input a Green
functor $\mR$, an $\mR$-bimodule $\mM$, and an element $g \in G$ which
is used to twist the action on $\mM$.

\begin{definition}
Let $\m{R}$ be a Green functor for $C_{n}$, and let $\mM$ be a left
$\mR$-module. If $g\in C_{n}$, then let $\gmM$ denote $\mM$ with the
module structure twisted by $g$: if $\mu$ is the action map for $\mM$,
then the action map ${}^{g}\!\mu$ for $\gmM$ is given by
\begin{center}
\begin{tikzcd}
\m{R}\Box \mM \arrow[rd, "{}^g\!\mu"] \ar[d, "g \Box 1"] & \\ \m{R}\Box\mM
\arrow[r, "\mu"] & \mM.
\end{tikzcd}
\end{center}
If $\mM$ is instead a right $\m{R}$-module, then we will denote the
obvious analogous structure by $\mMg$.

We call these induced module structures the {{twisted module
    structure}}.
\end{definition}

The (twisted) cyclic bar constructions we consider are simply the
evident generalization of the ordinary cyclic bar constructions where
we allow various twistings on the bimodule coordinate.  The (twisted)
cyclic cobar construction is defined analogously; see
Remark~\ref{rem:twistedcoh} for further discussion.

\begin{definition}
Let $G\subset S^{1}$ be a finite subgroup, let $g\in G$, and let $\mR$
be an associative Green functor for $G$, and let $\mM$ be an
$\mR$-bimodule. Define the {\em twisted cyclic nerve} of $\mR$ with
coefficients in $\gmM$ as the simplicial Mackey functor with
$k$-simplices given by
\[
[k] \mapsto
\HC{G}_{k}(\mR;\gmM):=\gmM\Box\mR^{\Box k}.
\]
For $1\leq i\leq k-1$, the face map $d_{i}$ is simply the
multiplication between the $i$\textsuperscript{th} and
$(i+1)$\textsuperscript{st} box factors. The $0$\textsuperscript{th}
face map $d_{0}$ is the ordinary (right) action map. The
$k$\textsuperscript{th} face map $d_{k}$ is given by the composite:
\[
\gmM\Box\mR^{\Box k}\xrightarrow{\tau} \mR\Box\gmM\Box\mR^{\Box (k-1)}\xrightarrow{{}^{g}\mu\Box Id} \gmM\Box\mR^{\Box (k-1)}.
\]
The map $\tau$ is an isomorphism rotating the last factor to the front. 
For $0\leq i\leq k-1$, the degeneracy map $s_{i}$ is induced by the
unit in the $(i+1)$\textsuperscript{st} factor.
\end{definition}

It is straightforward to verify that this is in fact a simplicial
object.  All identities that do not include the $k$th face map follow
from the standard arguments, relying only on the associativity of the
product and the definition of the unit.  The identities involving the
$k$th face map follow from the observation that the $G$-action is via
associative ring maps.

The standard homological algebra argument shows the following, since
$g$ in our definition only changed the bimodule, rather than changing
the fundamental category.
\begin{lemma}\label{lem:Tor}
For any finite $G\subset S^{1}$, any choice of $g\in G$, any
associative Green functor $\mR$ that is flat over the Burnside Mackey functor \(\mA\), and any $\mR$-bimodule $\mM$, the
homology of $\HC{G}_{\bullet}(\mR;\gmM)$ is
\[
\m{\boxtor}_{\ast}^{\mR\Box\mR^{op}}(\mR,\gmM).
\]
\end{lemma}

\begin{remark}
Flatness here is a harsher condition than classically, since even localizations need not preserve flatness. For example, for \(G=C_{2}\), inverting \(2\) in the underlying ring of the Burnside Mackey functor produces a localization which is not flat, splitting into a copy of the augmentation ideal of the Burnside Mackey functor (which is not flat) and the constant Mackey functor \(\Z[\tfrac{1}{2}]\). 
\end{remark}

\begin{remark}
It is obvious that we have similar constructions wherein we perturb
both the bimodule structures on $\m{R}$ and on $\mM$; we ignore these,
since they do not seem to play a role in topological Hochschild homology.
\end{remark}

\subsection{Twisted Hochschild Homology}\label{ssec:TwistedHH}

We now turn to our main definition, the construction of the twisted
Hochschild homology for Green functors, in terms of the twisted cyclic nerve.  The definition relies on the
fact that in Mackey functors for cyclic groups (or more generally, a
group in which all subgroups are normal), the Weyl group of $H$ is
always $G/H$.  This easy observation has the consequence that
$\mM(G/H)$ is actually a $G$-module for all subgroups $H$, and the
restriction and transfer maps are actually maps of $G$-modules.  In
particular, objects of $\Mack$ have an action of $G$.

\begin{definition}\label{generator}
Let $G \subset S^1$ be a finite subgroup. Throughout, we will refer to the element $g = e^{2\pi i / |G|} \in S^1$ as the \emph{generator} of $G$.
\end{definition}

For a cyclic group $G$, a choice of generator is equivalent to a
choice of embedding $G \to S^1$.

\begin{definition}
Let $g$ be the generator $e^{2\pi i / |G|} \in S^1$ of $G \subset S^1$. Let $\mR$ be an associative Green functor for $G$. The {\em $G$-twisted Hochschild
  homology}, $\m{\HH}_i^{G}(\m{R})$, is the homology of the twisted
cyclic nerve $\HC{G}_{\bullet}(\mR;{}^{g}\mR)$.
\end{definition}

The homology of a simplicial Mackey functor $\m{M}_{\bullet}$ is defined to be the homology of the associated
normalized dg Mackey functor (see Section \ref{sec:dgandSimpMackey}). When $\mR$ is a commutative Green functor, the twisted cyclic nerve inherits that
structure.

\begin{proposition}
If $\m{R}$ is a commutative Green functor, then $\HC{G}_{\bullet}(\mR)$ is a
simplicial commutative Green functor.
\end{proposition}

\begin{proof}
Since Green functors are the commutative monoids for the box product
of Mackey functors, all of the face maps except the last one are
obviously maps of Green functors. The degeneracy maps are just boxing
with the unit, and hence are also maps of Green functors. Finally,
since all of the structure maps in a Green functor are
$G$-equivariant, the action by an element $g$ is a map of Green
functors; the last face map is also a map of Green functors.
\end{proof}

Applying Lemma~\ref{lem:Tor} gives the following immediate
proposition.

\begin{proposition}\label{Prop:tor}
Let $G \subset S^1$ be a finite subgroup and $g\in G$ the
generator. For any $G$-Green functor $\m{R}$ which can be written as a filtered colimit of projectives, we have a natural
isomorphism
\[
\m{\HH}_i^{G}(\m{R})\cong \m{\boxtor}_i^{\mR\Box
  \mR^{op}}(\m{R},{}^{g}\m{R}),
\]
where $\m{\boxtor}_i$ is the $i$th derived functor of the box product.
\end{proposition}
\begin{proof}
This follows directly because filtered colimits of projectives are flat.
\end{proof}

The twisting by a choice of generator can be somewhat confusing. We
find it here most helpful to recast this in a more transparent way for
$\m{R}$ a commutative Green functor.

\begin{proposition}\label{prop:ZHomotopyOrbits}
Let $\m{R}$ be a commutative Green functor for $G$. Let $g\in G$ be the generator, and let $\mathbb Z\to G$ be the
associated homomorphism. Then the natural map 
\[
\mR\to
\HC{G}_{0}\mR
\] 
given by the inclusion of the zero cells extends
to a natural weak equivalence
\[
\m{R}_{h\mathbb Z}\xrightarrow{\simeq} \HC{G}_{\bullet}\mR
\]
between the homotopy orbits of $\m{R}$ in commutative Green functors
to the twisted Hochschild homology.
\end{proposition}

\begin{proof}
The formula for the boundary maps of the twisted Hochschild chains
shows that \(\HC{G}_{\bullet}(\mR)\) is exactly the homotopy coequalizer in
commutative Green functors of the identity with the generator \(g\) of
\(G\). By definition, this is also the homotopy orbits.
\end{proof}

\begin{remark}
The topological statement is also true, namely the same argument shows
that there is a natural equivalence between $R_{h\bZ}$ and $S^1
\otimes_G R$ for $R$ a commutative ring orthogonal
$G$-spectrum. Choosing different generators corresponds to tensoring
up over different embeddings $G \to S^1$.
\end{remark}

\begin{remark}\label{rem:twistedcoh}
In the equivariant context, what we might mean by topological
Hochschild cohomology becomes more subtle. Classically, this is the
endomorphism spectrum of $R$ as an $(R,R)$-bimodule, and the action on
$\THH$ is obvious. The equivariant structure of $\THH$, however, has an
asymmetry in the roles of $R$ in the two smash factors. In particular,
we see {\emph{both}} an action of $\End_{R^{e}}(R)$ and of
$\End_{R^{e}}(^{g}R)$.  We expect that the most relevant definition
will be guided by the applications; we leave further investigation of
this construction for future work.
\end{remark}

\section{Relative Hochschild homology for Green functors}\label{sec:Relative}

One motivation for studying Hochschild homology for Green functors will be its close relationship with relative topological Hochschild homology. In order to construct an algebraic analog of relative topological Hochschild homology, however, we need a relative version of Hochschild homology for Green functors. In other words, for $K \subset G \subset S^1$ finite subgroups, and $\m{R}$ a Green functor for $K$, we would like to define the $K$-relative $G$-twisted Hochschild homology of $\m{R}$, denoted $\m{\HH}^G_K(\m{R})_*$. 

In order to give the definition of relative twisted Hochschild homology for Green functors, we first need to discuss norms in Mackey functors.

\subsection{$G$-symmetric monoidal structure on $\Mack$}\label{ssec:GSymMackey}

Norms in Mackey functors are closely related to norms in orthogonal $G$-spectra, and we first present a definition of Mackey functor norms that illustrates this relationship. A reader who is less familiar with equivariant stable homotopy theory, however, may wish to start with Definition \ref{def:normhoy} instead.

Both the categories $\Sp_G$ and $\Mack$ have an equivariant enrichment
of the symmetric monoidal product, a $G$-symmetric monoidal category
structure.  Roughly speaking, a $G$-symmetric monoidal structure
consists of coherent multiplicative norms for all subgroups $H \subset
G$.  The $G$-symmetric monoidal structure on $\Sp_G$-spectra was
implicitly introduced by Hill-Hopkins-Ravenel~\cite{HHR} and then
later explicitly codified and studied in
Hill-Hopkins~\cite{HillHopkins}.  See also~\cite{BlumbergHill}
for a discussion for incomplete universes and~\cite{ClanBarwick} for
an $\infty$-categorical treatment.

In analogy with Remark~\ref{rem:Symm}, the $G$-symmetric monoidal
structure on $\Sp_G$ induces a $G$-symmetric monoidal category
structure on $\Mack$.  Specifically, we can define the norms in
$\Mack$ in terms of the norms in $\Sp_G$.

\begin{definition}[{\cite[Definition 5.9]{HillHopkins}}]\label{def:normhhr}
Let $K\subset G$ be a subgroup. If $\mM$ is a $K$-Mackey functor, then
let
\[
N_K^G\mM:=\m{\pi}_0^G (N_K^GH\mM).
\]
Here the superscript on $\m{\pi}_0$ is simply to help the reader keep track of the domain and range of the functor. 
\end{definition}

With this definition, the functor $\m{\pi}_0$ becomes a $G$-symmetric
monoidal functor from the category of $(-1)$-connected $G$-spectra to
Mackey functors. We can extend the norms to arbitrary finite $G$-sets
using the decomposition of a finite $G$-set into a disjoint union of
orbits:

\begin{definition}
If $H_1,\dots, H_n$ are subgroups of $G$ and $\mM$ is a $G$-Mackey
functor, then let
\[
N^{G/H_1\amalg\dots\amalg G/H_n} \mM := \big(N_{H_1}^G i_{H_1}^\ast
\mM\big)\Box\dots\Box \big(N_{H_n}^G i_{H_n}^\ast \mM\big).
\]
\end{definition}
For a finite $K$-set $T$, let $F_K(G,T)$ denote the $G$-set of $K$-equivariant maps from $G$ to $T$. The $G$-action is given by $(g f)(g') = f(g'g)$. 
\begin{proposition}\label{prop:normcalc}
For any finite $K$-set $T$, we have an isomorphism
\[
N_K^G(\mA^K_T)\cong\mA^G_{F_K(G,T)}.
\]
\end{proposition}

\begin{proof}
We have a natural equivalence of $K$-spectra
\[
H\mA^K_T\simeq \Sigma^{\infty}T_+\wedge H\mA^K.
\]
Since the norm is strong symmetric monoidal, we have equivalences
\[
N_K^G(H\mA^K_T) \htp N_K^G(\Sigma^{\infty}T_+ \sma H\mA^K) \cong
N_K^G(\Sigma^{\infty}T_+)\wedge N_K^G H\mA^K.
\]
By construction, we have an equivariant isomorphism
\[
N_K^G(\Sigma^{\infty} T_+)\cong \Sigma^{\infty} F_K(G,T)_+.
\]
The result follows once we identify $N_K^G \mA^K$. For this, we recall
that the Burnside Mackey functor $\mA$ is the zeroth homotopy group of
the zero sphere, so the fiber of the Postnikov section $S^0\to H\mA$
is $0$-connected. In particular, we know that
\[
\m{\pi}_0 N_K^G S^0\cong \m{\pi}_0 N_K^G H\mA^K.
\]
By construction $N_K^G S^0\simeq S^0$, from which the desired result
follows.
\end{proof}

Since the norm in spectra commutes with sifted colimits,
Proposition~\ref{prop:normcalc} can be used (very inefficiently) to
compute the norm of any Mackey functor.  However, a direct expression
for the norm of Mackey functors is often more useful.  The thesis work
of Mazur and Hoyer produced purely algebraic constructions of the norm
on Mackey functors~\cite{MazurArxiv,HoyerThesis}.  Mazur's
construction is analogous to Lewis' description of the box product of
Mackey functors and gives an explicit construction of the norm for
cyclic $p$-groups. Hoyer generalized this, giving a functorial
description of the norm for any finite group $G$.

\begin{definition}[{\cite[2.3.2]{HoyerThesis}}]\label{def:normhoy}
If $\mM$ is an $H$ semi-Mackey functor, then let $N_H^G\mM$ be the left Kan
extension of $\mM$ along the coinduction functor $F_H(G,-)$ on finite \(H\)-sets: 
\[
\xymatrix{ {\cA_{\N}^H}\ar[r]^{\mM}\ar[d]_{F_H(G,-)} & {\Set.}
  \\ {\cA_{\N}^G}\ar[ur]_{N_H^G\mM} }
\]
\end{definition}

\begin{proposition}[{\cite[2.3.6]{HoyerThesis}}]
If \(\mM\) is an \(H\)-Mackey functor, then \(N_{H}^{G}\mM\) is a \(G\)-Mackey functor.
\end{proposition}

\begin{remark}
Although coinduction is the right adjoint to the forgetful functor on
finite $G$-sets, it is not so on the Burnside category. Moreover, it
is not a product preserving functor on the Burnside category. In
particular, the Hoyer norm is not a left adjoint on Mackey functors.
\end{remark}

Since the Hoyer norm is a left Kan extension, we again have a canonical identification on representable functors. Composing with group completion then gives the following.
\begin{proposition}\label{prop:NormofProjectives}
If T is a finite \(H\)-set, then we have a canonical isomorphism of Mackey functors
\[
N_H^G \mA^H_T\cong \mA^G_{F_H(G,T)}.
\]
\end{proposition}
This agrees with the construction in terms of the norm on spectra, as both are expressing a distributive law.
Hoyer showed that this is true very generally, building on work of
Ullman.

\begin{theorem}[{\cite[2.3.7]{HoyerThesis}}]\label{thm:hoyer}
There is a canonical isomorphism between $N_H^G$ as defined in
Definition~\ref{def:normhhr} and Definition~\ref{def:normhoy} on
Mackey functors and on Green functors.
\end{theorem}

From this perspective, it is straightforward to verify the properties
of the $G$-symmetric monoidal structure on $\Mack$.

\begin{proposition}\label{prop:Composite}
If $K\subset H\subset G$, then we have a natural isomorphism of
functors
\[
N_H^G\circ N_K^H\cong N_K^G.
\]
\end{proposition}

\begin{proof}
Since we have a natural isomorphism of functors
\[
F_H\big(G,F_K(H,-)\big)\cong F_K(G,-),
\]
we have a natural isomorphism of the corresponding left Kan
extensions.
\end{proof}

We can similarly understand the composition of the restriction functor
with the norm.  This is closely connected with the external
description of a Tambara functor as a $G$-commutative monoid in the
category of Mackey functors \cite{MazurArxiv, HoyerThesis}, as it describes the
externalized Weyl action. First we need a standard, categorical definition.

\begin{definition}
If \(K\) is a finite group, then let \(BK\) be the category with one object \(\ast\) and with the morphisms from \(\ast\) to itself given by \(K\), with composition the multiplication in \(K\).

If \(\cC\) is a category, then a \(K\)-object in \(\cC\) is a functor \(BK\to\cC\).
\end{definition}
In particular, a \(K\)-object in \(G\)-Mackey functors is a \(G\)-Mackey functor \(\mM\) such that for each subgroup \(H\), \(\mM(G/H)\) is a \(K\)-module, and all restriction, transfer, and conjugation maps are maps of \(K\)-modules. A \(K\)-object in \(G\)-Green functors moreover requires the multiplication be a \(K\)-equivariant map.

\begin{proposition}\label{prop:LanGmodH}
If $\m{M}$ is a $G$-Mackey functor, then for all $H\subset G$,
$N_H^Gi_H^\ast\mM$ is a $W_G(H)$-object in $G$-Mackey functors.

If $\m{R}$ is a $G$-Green functor, then $N_H^G i_H^\ast \m{R}$ is a
$W_G(H)$-object in $G$-Green functors.
\end{proposition}

\begin{proof}
By definition, the functors $N_H^G$ and $-\Box-$ are left Kan
extensions. The functor $i_H^\ast$ is so too: $i_H^\ast$ is the left
Kan extension of the restriction functor $i_H^\ast$ on the Burnside
category. This means we have a natural isomorphism of functors
\[
N_H^Gi_H^\ast(-)=\Lan_{F_H(G,-)}\Lan_{i_H^\ast}(-)\cong
\Lan_{F_H(G,i_H^\ast(-))}(-).
\]
We have a natural isomorphism of functors
\[
F_H(G,i_H^\ast(-))\cong F(G/H,-),
\]
and this isomorphism is compatible with the isomorphisms showing that
$F_H(G,-)$ and $F(G/H,-)$ are strong symmetric monoidal functors. In
particular, the isomorphism of left Kan extensions on Mackey functors
gives an isomorphism of left Kan extensions on Green functors.

The automorphism group $Aut_G(G/H)\cong W_G(H)^{op}$ acts on the right
on the functor $\Lan_{F(G/H,-)}(-)$ via precomposition, and this gives
us a natural action of $W_G(H)$ on the composite functor, as desired.
\end{proof}

Combining Proposition~\ref{prop:Composite} and
Proposition~\ref{prop:LanGmodH} yields the following corollary.

\begin{corollary}
If $\mM$ is a $G$-Mackey functor and we have a chain of groups
$H\subset G\subset G'$, then the $G'$-Mackey functor
\[
N_H^{G'} i_H^\ast\mM
\]
is naturally a $W_G(H)$-object. If $\mM$ is a Green functor, then this
is again a $W_G(H)$-object in Green functors.
\end{corollary}

This is conceptually connected to our later construction of Hochschild
homology in this context. The following proposition is immediate.

\begin{proposition}\label{prop:Restriction}
For any ring $R$ and for any group $G$, there is an associative
$G$-Green functor $\m{R}$ for which $i_e^{\ast}\m{R}\cong R$.
\end{proposition}
\begin{proof}
Endow $R$ with the trivial $G$-action and let $\m{R}$ be the
fixed-point Mackey functor. Since $R$ is an associative ring, this is
an associative Green functor.
\end{proof}

\begin{corollary}
For any associative ring $R$, the associative Green functor $N_e^G R$
is a $G$-object in Green functors.
\end{corollary}
\begin{proof}
Proposition~\ref{prop:Restriction} shows that there is an associative
Green functor $\m{R}$ such that $i_e^{\ast}\mR\cong R$ as associative
rings (which are the same as associative Green functors for the
trivial groups). In particular, we know that
\[
N_e^G R\cong N_e^G i_e^\ast \m{R},
\]
so by Proposition~\ref{prop:LanGmodH}, this is a $W_G(e)=G$-object in
$G$-Green functors.
\end{proof}

\begin{remark}
If $R$ is an associative ring orthogonal spectrum, then $N_e^G R$ is
naturally an associative ring orthogonal $G$-spectrum. The proof is
essentially the same: the role of the fixed point Mackey functor is
played here by the push-forward functor $\Sp\to\Sp^G$.
\end{remark}

The other composite of the restriction and norm is also easy to
understand.  For simplicity, we restrict attention here to abelian
groups, where all subgroups are normal. In this case, if \(J,H\subset G\) are subgroups, then the double cosets are the same as the cosets of \(JH\).

\begin{theorem}\label{thm:ResofNorm}
Let $G$ be a finite, abelian group and let \(H\) and \(J\) be subgroups. Then we have a natural isomorphism
\[
i_{J}^{\ast}N_{H}^{G}(\mhyphen)\cong\big(N_{H\cap J}^{J}i_{H\cap J}^{\ast}(\mhyphen)\big)^{\Box [G:JH]}.
\]
The action of \(G\) on this is via tensor induction.
\end{theorem}
\begin{proof}
The restriction in Mackey functors can also be modeled by a left Kan extension: it is the left Kan extension along the ordinary restriction functor from \(\cA^{G}\) to \(\cA^{J}\). It therefore suffices to show that we have a natural isomorphism of functors
\[
i_{J}^{\ast} F_{H}(G,\mhyphen)\cong \prod_{g\in J\backslash G/H} F_{H\cap J}(J,i_{H\cap J}^{\ast}(\mhyphen)\big).
\]
Here we observe that we are forming the indexed product in the sense of \cite[\S A.3]{HHR}, so it suffices to show this at the level of diagrams. The result follows from the observation that
\[
i_{J}^{\ast}G/H\cong \coprod_{g\in J\backslash G/H} J/(J\cap H).
\]
\end{proof}

\begin{corollary}
Let $G$ be an abelian group.
\begin{enumerate}
\item For $H\subset J\subset G$, we have a natural isomorphism of
  functors
\[
i_J^\ast N_H^G\cong \big(N_H^J\big)^{\Box [G:J]}.
\]
\item For $J\subset H\subset G$, we have a natural isomorphism of
  functors
\[
i_J^\ast N_H^G\cong \big(i_J^\ast\big)^{\Box [G:H]}.
\]
\end{enumerate}

\end{corollary}

We include a small computation that is of independent interest.  We
point the reader to Section~\ref{sec:Example} for further discussion
of the computation here.

\begin{proposition}
Let $G=C_{p^{n}}$ and let $M$ be a module over an $\mathbb
F_{p}$-algebra $R$. Then the Mackey functor $N_{e}^{C_{p^{n}}}M$ is a
module over the constant Mackey functor $\m{\Z}$: for any $1\leq k\leq
n$, we have
\[
tr_{C_{p^{k-1}}}^{C_{p^{k}}}\circ res_{C_{p^{k-1}}}^{C_{p^{k}}}=p.
\]
\end{proposition}
\begin{proof}
Since the norm functor is strong symmetric monoidal, it suffices to
show this for $M=R=\mathbb F_{p}$. Here Mazur's original construction
and her example, \cite[Example 3.12]{MazurArxiv}, show that the norm to $C_{p^{n}}$ of
$\mathbb F_{p}$ has
\[
(N_{e}^{C_{p^{n}}}\mathbb F_{p})(C_{p^{n}}/C_{p^{k}})=\mathbb
Z/p^{k+1},
\]
the restriction maps are all surjective, and the transfer maps are the
obvious inclusions.
\end{proof}

\subsection{Relative Twisted Hochschild homology}
We are now ready to define relative versions of the twisted cyclic nerve and the Hochschild homology of Green functors.

\begin{definition}
Let $K\subset G\subset S^1$ be finite subgroups, let $g$ be the
generator $e^{2\pi i / |G|} \in S^1$ of $G \subset S^1$, and let $\mR$ be an associative Green functor for
$K$. Define the $G$-twisted cyclic nerve relative to $K$ as
\[
\HC{G}_{K}(\mR)_{\bullet}:=\HC{G}_{\bullet}\big(N_K^{G}\mR;{}^{g}(N_{K}^{G}\mR)\big).
\]
The $G$-twisted Hochschild homology relative to $K$,  $\m{\HH}^{G}_K(\m{R})_{\ast}$, is the homology
of $\HC{G}_K(\mR)_{\bullet}$.
\end{definition}

Applying Lemma~\ref{lem:Tor} gives the following immediate
proposition, the relative version of Proposition \ref{Prop:tor}.

\begin{proposition}
Let $K \subset G \subset S^1$ be finite subgroups and $g\in G$ the
generator. For any $K$-Green functor $\m{R}$ which can be written as a filtered colimit of projectives, we have a natural
isomorphism
\[
\m{\HH}_K^{G}(\m{R})_i\cong \m{\boxtor}_i^{N_K^G\mR\Box
  N_K^G\mR^{op}}(N_K^G\m{R},{}^{g}(N_K^G\m{R})),
\]
where $\m{\boxtor}_i$ is the $i$th derived functor of the box product.
\end{proposition}
\begin{proof}
The norm preserves projective objects and filtered colimits, so in particular, the norm of \(\m{R}\) is again a filtered colimit of projectives. These are flat.
\end{proof}

\section{Differential graded and simplicial Mackey functors}\label{sec:dgandSimpMackey}

When discussing the homotopical nature of our various functors, we use
model structures on the categories of dg and simplicial Mackey
functors.  Regarding the category of Mackey functors as modules over a
ring, establishing the existence of such model structures is standard;
we take as a convenient reference~\cite{GoerssSchemmerhorn}, although
of course these results go back to Quillen~\cite[\S 2]{Quillen}.

\begin{theorem}\label{thm:dgmackmod}
The category $\Ch_*^{+}(\Mack)$ of non-negatively graded dg Mackey functors has
a model structure in which a map $f \colon \m{B}_{\bullet} \to
\m{C}_{\bullet}$ is
\begin{enumerate}
\item a weak equivalence if it induces a quasi-isomorphism (i.e., an
  isomorphism of homology Mackey functors),
\item a fibration if it is a surjection in positive degrees, and
\item a cofibration if it is an injection with projective levelwise
  cokernel.
\end{enumerate}
\end{theorem}

We can still obtain a model structure on the category $\Ch_*(\Mack)$
of unbounded complexes of Mackey functors~\cite[2.3.11]{Hovey}, but in
that case the description of the cofibrations becomes more
complicated.  For our purposes, however, it suffices that cofibrant
objects in $\Ch_*(\Mack)$ are in particular levelwise
projective~\cite[2.3.9]{Hovey}.

\begin{remark}
Generalizing the comparison between $H\bZ$-module spectra and dg
modules over $\bZ$ (e.g., see~\cite{SchwedeShipley, Shipley}), there
is a Quillen equivalence between the category $\Ch_*(\Mack)$ and the
category of modules in $\Sp^G$ over $H\m{A}$.  However, a careful
proof of this result has not yet appeared in the literature.  Note 
also that there is some subtlety to the multiplicative story, as
recorded in~\cite{Ullman}: if \(\mR\) is a commutative Green functor
that has no Tambara functor structure, then \(H\mR\) cannot be a
commutative ring spectrum.
\end{remark}

We now turn to the category $\sMack$ of simplicial Mackey functors.
The Dold-Kan correspondence shows that the category $\sMack$ is
equivalent to the category $\Ch_*^{+}(\Mack)$ via the normalized chain
complex construction; this provides a lifted model structure on
$\sMack$~\cite[4.4.2]{GoerssSchemmerhorn}.

\begin{theorem}\label{thm:simpmackmod}
The category $\sMack$ of simplicial Mackey functors has a simplicial
model structure in which a map $f \colon \m{M}_{\bullet} \to
\m{N}_{\bullet}$ is
\begin{enumerate}
\item a weak equivalence if it induces a quasi-isomorphism of
  associated normalized chain complexes,
\item a fibration if it is a fibration of simplicial sets, and
\item a cofibration if the cokernel of the underlying degeneracy
  diagrams (i.e., the objects obtained by forgetting the face maps) is
  of the form
\[
\bigoplus_k \bigoplus_{\phi \colon [n] \to [k]} \phi^* \m{P}_k
\]
in degree $n$, where the second sum runs over the surjections and each $\m{P}_k$ is projective.
\end{enumerate}
\end{theorem}

We also have several easy but very useful observations about this
model structure; the first is immediate, and the second follows from
the usual Dold-Kan analysis.

\begin{proposition}\label{prop:cofibrant}
\mbox{}
\begin{enumerate}
\item Every simplicial Mackey functor is fibrant.
\item A simplicial Mackey functor is cofibrant if and only if it is levelwise projective.
\end{enumerate}
\end{proposition}

As suggested by the description of the weak equivalences in
Theorem~\ref{thm:simpmackmod}, we make the following definition.

\begin{definition}
Let $\m{M}_{\bullet}$ be a simplicial Mackey functor.  The homology
$H_*(\m{M}_{\bullet})$ is defined to be the homology of the associated
normalized dg Mackey functor.
\end{definition}

\begin{remark}
Note that in contrast to the notion of dg Mackey functor that is
relevant in the context of Kaledin's work~\cite{KaledinMackey,
  KaledinICM}, here we are simply using dg objects in the category of
Mackey functors.  Kaledin instead studies dg functors out of a dg
model of the Burnside category, i.e., the algebraic analogue of the Guillou--May and Barwick
``spectral presheaves'' approach to the equivariant stable category \cite{GuillouMay} \cite{Barwick}.
\end{remark}

We now turn to investigate the interaction of the box product and the norms with
these model structures.
\subsubsection{The derived box product}
\begin{definition}
If \(\mM_\bullet\) and \(\m{N}_\bullet\) are simplicial Mackey functors, then let \(\mM_\bullet\Box\m{N}_\bullet\) be the simplicial Mackey functor with 
\[
(\mM_\bullet\Box\m{N}_\bullet)_k=\mM_k\Box\m{N}_k,
\]
and where the structure maps are just the box products of the corresponding structure maps.
\end{definition}

Since cofibrant objects in the model structure of
Theorem~\ref{thm:simpmackmod} are levelwise projective, the following
proposition is immediate.

\begin{proposition}
Let $\mP$ be a cofibrant simplicial Mackey functor.  Then the functor
$\mP \Box (-)$ preserves weak equivalences and cofibrant objects.
\end{proposition}

Similarly, the box product induces a symmetric monoidal product on
$\Ch_*^{+}(\Mack)$ and $\Ch_*(\Mack)$ in the usual fashion; again we
use the description of the cofibrant objects in the model structures
to prove the following result.

\begin{proposition}
Let $\mP_{\cdot}$ be a cofibrant dg Mackey functor.  Then the functor
$\mP_{\cdot} \Box (-)$ preserves weak equivalences.
\end{proposition}

\begin{remark}
One might wonder about the multiplicative comparison between dg Mackey
functors and simplicial Mackey functors.  Classically, only one of the
two maps connecting simplicial abelian groups and non-negatively
graded chain complexes is appropriately monoidal: the normalized chain
functor is lax symmetric monoidal while the inverse equivalence is only
\(E_{\infty}\), as described by Richter \cite{RichterChZsAb}.  We
believe that a similar situation holds here, with additional complexity
introduced by the subtlety of equivariant multiplicative structure.
If $\mP$ is a simplicial Green functor, a suitable generalization of
Richter's argument should show that the associated dg Mackey functor
is an $E_\infty$ monoid for the trivial $E_\infty$ operad. We believe
that extending this result to simplicial Tambara functors, or general
$N_\infty$ algebra objects, will require additional structure to
encode the norm operation on chain complexes of Mackey functors.

\end{remark}

\subsubsection{The derived norm}
We can similarly define the norm levelwise and check that it is derivable. We will use this in \S~\ref{sec:Shukla} to build a homotopy invariant form of our complexes. Of course, if we step through spectra, then we know that the norm is derivable, since the norm in spectra is a left Quillen functor. We give a more self-contained treatment here.

\begin{definition}
If \(\mM_\bullet\) is a simplicial \(H\)-Mackey functor, then let \(N_H^{G}\mM_\bullet\) be the simplicial Mackey functor for \(G\) defined by applying the norm functor levelwise.
\end{definition}

We being with a definition motivated by Bredon homology. Using the transfer, the maps
\[
\mA_{(\mhyphen)}\colon T\mapsto \mA_T
\]
extend to a functor from the category of finite \(G\)-sets to Mackey functors. This gives a functor from simplicial \(G\)-sets with finitely many \(k\)-simplices for all \(k\) to simplicial Mackey functors.

\begin{definition}
Let \(T_\bullet\) be a simplicial \(G\)-set such that for each \(k\), the set \(T_k\) is finite. Composing with \(\mA_{(\mhyphen)}\) then gives a simplicial Mackey functor, which we will denote \(\mA\cdot T_\bullet\). This gives a functor from the full subcategory of finite simplicial \(G\)-sets to \(\sMack\).
\end{definition}

\begin{remark}
For \(T_\bullet\) as above, the complex associated to \(\mA\cdot T_\bullet\) is the ordinary cellular chain complex computing Bredon homology of the geometric realization of \(T_\bullet\) with coefficients in $\mA$.
\end{remark}

Coinduction provides a kind of norm from simplicial \(H\)-sets to simplicial \(G\)-sets.

\begin{definition}
If \(T_\bullet\) is a finite simplicial \(H\)-set, then let 
\[
\Map_H(G,T_\bullet)
\]
be the finite simplicial \(G\)-set which arises by composing with the coinduction functor.
\end{definition}

Since the norm on a representable Mackey functor is easy to compute, the following is immediate.
\begin{proposition}\label{prop:SimplicialNormofFree}
For any finite simplicial \(H\)-set \(T_\bullet\), we have a natural isomorphism of simplicial \(G\)-Mackey functors
\[
N_H^G\big(\mA^H\cdot T_\bullet\big)\cong \mA^G\cdot \Map_H(G,T_\bullet).
\]
\end{proposition}

Homotopy invariance of the norm then follows from two observations.

\begin{proposition}
Let \(\mM_\bullet\) and \(\m{N}_\bullet\) be simplicial Mackey functors and let
\[
f, f'\colon\mM_\bullet\to \m{N}_\bullet
\]
be two maps of simplicial Mackey functors. Then a simplicial homotopy \(F\) from \(f\) to \(f'\) is a map of simplicial Mackey functors
\[
\mM_\bullet\Box\big( \mA\cdot\Delta^1_\bullet\big)\to\m{N}_\bullet
\]
which restricts to \(f\oplus f'\) on
\[
\mM_\bullet\oplus\mM_\bullet\cong\mM\Box\big(\mA\cdot\partial \Delta^1_\bullet\big).
\]
\end{proposition}

\begin{lemma}\label{lem:NormPreservesHomotopies}
The norm preserves simplicial homotopies.
\end{lemma}
\begin{proof}
The norm functor, being strong symmetric monoidal, gives a map 
\[
N_H^G(F)\colon N_H^G\mM_\bullet\Box N_H^G\big(\mA\cdot\Delta^1_\bullet\big)\to N_H^G\m{N}_\bullet.
\]
We can rewrite the source using Proposition~\ref{prop:SimplicialNormofFree}:
\[
N_H^G\big(\mA\cdot\Delta^1_\bullet\big)\cong\mA\cdot\Map_H(G,\Delta^1_\bullet).
\]
Restricting along the diagonal map
\[
\Delta^1_\bullet\hookrightarrow\Map_H(G,\Delta^1_\bullet)
\]
then gives the desired simplicial homotopy.
\end{proof}

\begin{corollary}\label{cor:NormHomotopical}
The norm on simplicial Mackey functors takes cofibrant objects to cofibrant objects and preserves weak equivalences between cofibrant objects.
\end{corollary}
\begin{proof}
Proposition~\ref{prop:cofibrant} shows that a simplicial Mackey functor is cofibrant if and only if it is levelwise projective, and Proposition~\ref{prop:NormofProjectives} shows that the norm preserves the generating projectives. This shows that the norm on simplicial Mackey functors preserves cofibrant objects.

For the second part, any weak equivalence between cofibrant simplicial Mackey functors is a homotopy equivalence. Lemma~\ref{lem:NormPreservesHomotopies} shows that the norm preserves these.
\end{proof}

\section{Relationship to topological Hochschild homology}\label{sec:THH}

Classically, for a ring $A$ the Dennis trace map from the algebraic $K$-theory of $A$ to the Hochschild homology of $A$ factors  through the topological Hochschild homology of the Eilenberg-MacLane spectrum $HA$:
$$
K_k(A) \to \pi_k(\THH(HA)) \to \HH_k(A).
$$
The map from THH to its algebraic analog HH is called linearization, and this linearization map is an isomorphism in degree 0:
$$
\xymatrix{
 \pi_0(\THH(HA)) \ar[r]^{\cong} &\HH_0(A).}
$$
More generally, for a $(-1)$-connected ring spectrum $R$ there is a linearization map
$$
\xymatrix{
 \pi_k(\THH(R)) \ar[r]&\HH_k(\pi_0R)},
$$
which is again an isomorphism in degree 0. 

Let $H \subset S^1$ be a finite subgroup. For an $H$-equivariant ring spectrum $R$, one can form the $H$-relative THH of $R$, THH$_H(R)$, as defined in \cite{normTHH}. In this section we prove that our Hochschild homology for Green functors is an algebraic analog to this relative topological Hochschild homology. For instance, we will prove that it receives a linearization map from relative THH, and that this map is an isomorphism in degree 0. 

In approximating algebraic $K$-theory, the equivariant homotopy groups of topological Hochschild homology also play a crucial role. Our new constructions are algebraic approximations to these topological theories. We make this relationship precise in this section.

\subsection{Comparison to topological Hochschild homology}\label{ssec:ComptoTHH}
 
 A key connection between topological Hochschild homology and our new algebraic constructions comes from the close relationship between the algebraic
norm in Mackey functors and the Hill-Hopkins-Ravenel norm in spectra
(recall Theorem~\ref{thm:hoyer}).  This correspondence allows us to
establish a tight connection between the twisted cyclic nerve and
the $H$-relative topological Hochschild homology $\THH_H$ considered in
\cite{normTHH}.

\begin{theorem}\label{thm:ConnectiontoTHH}
Let $H\subset G\subset S^1$ be finite subgroups and let $R$ be a
$(-1)$-connected commutative ring orthogonal $H$-spectrum. We have a
natural isomorphism
\[
\m{\pi}_0^G\big(\THH_H\!(R)\big)\cong \m{\HH}_H^G(\m{\pi}_0^H R)_0.
\]
\end{theorem}

\begin{proof}
For any $H$-commutative ring spectrum $R$, we have a natural weak
equivalence
\[
i_G^\ast \THH_H\!(R)\simeq (N_H^G R)\sm{N_H^G R^e} {}^{g}(N_H^G R),
\]
where just as above, $^g(N_H^G R)$ is just $N_H^G R$ with a bimodule
structure twisted on one side by the Weyl action. If $R$ is
$(-1)$-connected, then so is $N_H^G R$, and for degree reasons the
K\"unneth spectral sequence gives
\[
\m{\pi}_0^G\Big((N_H^G R)\sm{N_H^G R^e} {}^{g}(N_H^G R)\Big)\cong
\big(\m{\pi}_0^G N_H^G R\big)\boxover{\big(\m{\pi}_0^G N_H^G
  R\big)^{e}} {}^{g}\big(\m{\pi}_0^G N_H^G R\big),
\]
which is exactly $\m{\HH}_H^G \big(\m{\pi}_0R\big)_0$.
\end{proof}

More generally, we have a natural homomorphism from the $H$-relative \(TR\) groups to our twisted Hochschild homology, refining the classical trace map.

\begin{theorem}\label{thm:THHEdgeMap}
For $H\subset G \subset S^1$ finite subgroups, and $R$ a $(-1)$-connected $H$-equivariant commutative ring spectrum, we have a natural homomorphism
\[
\underline{\pi}^G_{k} \THH_{H}(R)\to \m{\HH}_{H}^{G}(\m{\pi}^{H}_0R)_k.
\]
\end{theorem}

\begin{proof}
Recall that for any simplicial spectrum \(E_{\bullet}\), we have a spectral sequence 
\[
E^{2}_{p,q}=H_{p}\big(\pi_{q}(E_{\bullet}\big)\Rightarrow \pi_{p+q}\big(|E_{\bullet}|\big),
\]
where here we have simply applied \(\pi_{q}\) to \(E_{\bullet}\) to get a simplicial abelian group \cite[X.2.9]{EKMM}. The differentials are homology Serre type, and hence if for all \(k\), \(E_{k}\) is \((-1)\)-connected, then we have an edge homomorphism
\[
\pi_{p}\big(|E_{\bullet}|\big)\to H_{p}\big(\pi_{0}(E_{\bullet})\big).
\]

The proof is by the filtration by the skeleton, and goes through without change equivariantly. Applied to our context, we then have a spectral sequence with
\[
E^{2}_{p,q}=H_{p}\Big(\m{\pi}_{q}\big((N_{H}^{G}R)^{\wedge (\bullet+1)}\big)\Big)\Rightarrow \m{\pi}_{p+q}\THH_{H}(R).
\]
The edge homomorphism is
\[
\m{\pi}_{p}\THH_{H}(R)\to H_{p}\Big(\m{\pi}_{0}\big((N_{H}^{G}R)^{\wedge (\bullet+1)}\big)\Big).
\]
The collapse of the K\"unneth spectral sequence in degree zero, together with the definition of the norm in Mackey functors, then shows that we have an isomorphism of simplicial Green functors
\[
\m{\pi}_{0}\big((N_{H}^{G}R)^{\wedge (\bullet+1)}\big)\cong \HC{G}_{H}(\m{\pi}_{0}R)_{\bullet},
\]
as desired. 
\end{proof}

In the case where $H$ is the trivial group, and hence $R$ is a ring, the linearization map of Theorem~\ref{thm:THHEdgeMap} yields new lifts of the Dennis trace. 

\begin{corollary}\label{cor:TraceMap}
For $R$ a commutative ring, and any finite \(G\subset S^{1}\), we have a lift of the Dennis trace
\[
K_{q}(R)\to \m{\HH}_{e}^{G}(R)(G/G)_q.
\]

\end{corollary}
\begin{proof} The following diagram commutes, where the top is the composite of  the B\"okstedt-Hsiang-Madsen trace map and the edge map from Theorem~\ref{thm:THHEdgeMap}. The bottom composite is the classical Dennis trace. 
\[
\xymatrix{& \pi_q(\THH(R)^G)\cong \m{\pi}_q^G\THH(R)(G/G) \ar[r] \ar[dd]^{res_e^G} &  \m{\HH}_{e}^{G}(R)_q(G/G) \ar[dd]^{res_e^G} \\
\pi_{q}K(R) \ar[ur]^{trc} \ar[dr]& & \\
& \pi_q\THH(R)\cong  \m{\pi}_q^G\THH(R)(G/e) \ar[r] &\HH_q(R) \cong \m{\HH}_e^G(R)_q(G/e)
}
\]
\end{proof}

The topological Dennis trace map from algebraic $K$-theory to THH lifts through another invariant of rings (or ring spectra), topological cyclic homology, TC: 
$$
K(R) \to \textup{TC}(R) \to \THH(R). 
$$
Defining topological cyclic homology from THH relies on the fact that THH(R) is a cyclotomic spectrum. Roughly, an $S^1$-spectrum $X$ is cyclotomic if for all $n$ there are compatible equivalences of $S^1$-spectra $r_n: \rho_n^*\Phi^{C_n} X \to X$, where $\Phi^{C_n}$ denotes $C_n$ geometric fixed points. Here $\rho_n: S^1 \cong S^1/C_n$ is the $n$th root isomorphism.  (see \cite{BMcyclo} for more details about cyclotomic spectra).

Our twisted cyclic nerve also exhibits a type of cyclotomic structure. In order to study this, we first need to define a notion of geometric fixed points for Mackey functors.

\subsection{Geometric Fixed Points in Mackey Functors}\label{sec:GeomFP}

Our construction of the twisted Hochschild complex of a Green functor
will have a ``cyclotomic'' structure, relating the geometric fixed
points to the original complex.  To explain what we mean by cyclotomic
here, we must first describe geometric fixed points in the derived
category of Mackey functors.  This is related to Remark 6.5 of \cite{HillSlice} and B.7 of \cite{Barwick}. We will define our version in terms of a particular
choice of point-set model for the derived functor.

The material in this section works for all finite groups; we let $\mA$
denote the Burnside Mackey functor for the finite group $G$.

\begin{definition}\label{def:TildeEFofA}
Fix a finite group $G$ and let $N \subset G$ be a normal subgroup.
Let $\EF_{N}(\mA)$ denote the sub-Mackey functor of $\mA$ generated by
$\mA(G/H)$ for all subgroups $H$ which do not contain $N$.  Finally,
let
\[
\tildeEF_{N}(\mA)=\mA/\EF_{N}(\mA).
\]
\end{definition}

If $H$ does not contain $N$, then by construction
$\tildeEF_{N}(\mA)(G/H)=0$, while if $H$ does contain $N$, then
$\tildeEF_{N}(\mA)(G/H)$ is the quotient of $\mA(G/H)$ by the
transfers from any subgroup that does not contain $N$. In other words,
\[
\tildeEF_{N}(\mA)(G/H)=\begin{cases} 0 & N\not\subset H
\\ \mA\big((G/N)/(H/N)\big) & N\subset H.
\end{cases}
\]

Our definition of the geometric fixed points will be a composite; we
begin by taking the box product with $\tildeEF_{N}$.

\begin{definition}\label{def:TildeEF}
If $\mM$ is a $G$-Mackey functor and $N$ is a normal subgroup of $G$,
then let
\[
\tildeEF_{N}\mM=\mM\Box \tildeEF_{N}\mA.
\]

If $\mM_{\bullet}$ is a dg-$G$-Mackey functor, then let
\[
(\tildeEF_{N}\mM)_{n}=\mM_{n}\Box \tildeEF_{N}\mA,
\]
with the obvious differential.
\end{definition}

This is closely connected to the topological object by the same
name. Associated to $N$ is a smashing localization which nullifies all
spectra induced up from subgroups that do not contain $N$, the value
of which on the sphere is commonly denoted $\tildeEF_{N}$. The
following is immediate from direct considerations of the nullification
functor.

\begin{proposition}\label{prop:tildeEF}
For any Mackey functor $\mM$, we have
\[
\m{\pi}_0 (\tildeEF_N \sma H\mM)\cong \tildeEF_N\mM,
\]
where here $\tildeEF_N$ on the left denotes a cofibrant model of the
universal space for the family of subgroups not containing $N$.
\end{proposition}

\begin{proof}
By Proposition~\ref{prop:pizerosym}, we know that we have an
isomorphism
\[
\m{\pi}_0 (\tildeEF_N\wedge H\mM)\cong \m{\pi}_0(\tildeEF_N)\Box \mM.
\]
The space $\tildeEF_N$ is the localization of the zero sphere where we
kill the localizing subcategory generated by $G/H_+$ for
$H$ a subgroup not containing $N$. The zeroth homotopy Mackey functor
of $\Sigma^{\infty}G/H_+$ is by construction
\[
\mA^G_{G/H}.
\]
Since $G/H_+$ agrees with $\mathrm{Ind}_H^G S^0$, we know also that the set of
equivariant homotopy classes of maps
\[
G/H_+\to X
\]
is the group $\m{\pi}_0(X)(G/H)$. In particular, we see that
$\m{\pi}_0$ of the nullification of $S^0$ killing these $G/H_+$ agrees
with the nullification in Mackey functors.  This is exactly the
quotient.
\end{proof}

\begin{remark}
One might think that a good model of $\tildeEF_{N}$ as dg Mackey
functor could be obtained by taking the cellular chains on the
$G$-space $S^{\infty\bar{\rho}}$, where $\bar{\rho}$ is the reduced
regular representation of $G/N$.  However, although there is an
isomorphism
\[
\m{\pi}_0 (\m{C}_* S^{\infty\bar{\rho}}) \cong \tildeEF_N,
\]
the canonical map is {\em not} in general a quasi-isomorphism. For example, if \(G=C_{p}\), then 
\[
\m{\pi}_{2}\big(\m{C}_{\ast} S^{\infty\bar{\rho}}\big)=\m{H}_{2}\big(S^{\infty\bar{\rho}}\big)\cong \tildeEF_{G}/p
\]
is the reduction modulo \(p\) of \(\tildeEF_{G}\).
\end{remark}

The proof of Proposition~\ref{prop:tildeEF} explains how we should
compute our derived geometric fixed points.  Specifically, we want the
nullification in the category of dg Mackey functors which kills the
localizing, triangulated subcategory generated by $\mA_{G/H}$ for $H$
not containing $N$.  From the definition of $\tilde{E}F_{N}\mA$, we
immediately deduce the following proposition which lets us compute
this nullification.

\begin{proposition}\label{prop:pullback}
For any normal subgroup $N$ and for any $dg$-Mackey functor
$\mM_\bullet$, the Mackey functor $\tildeEF_{N}(\mM)_\bullet$ is in
the image of the (fully faithful) pullback $\pi_{N}^{\ast}$ from
$G/N$-$dg$-Mackey functors to $G$-$dg$-Mackey functors.
\end{proposition}

\begin{remark}
The preceding proposition is closely related to an observation of
Kaledin that homology is not sufficient to detect the image of the
pullback (see the introduction to~\cite{KaledinMackey}).
\end{remark}

Using Proposition~\ref{prop:pullback} and the fact that the pullback
is an isomorphism onto its image, we can now define the geometric
fixed points.

\begin{definition}\label{def:GFP}
For a normal subgroup $N$, let the {\em geometric fixed points} of a
Mackey functor $\mM$ be defined as
\[
\Phi^{N}(\mM)={\pi_{N}^{\ast}}^{-1}\big(\tilde{E}F_{N}(\mM)\big).
\]
If $\mM_{\bullet}$ is a dg-Mackey functor, then $\Phi^{N} \mM_{\bullet}$ is obtained by applying $\Phi^{N}$ levelwise. 
\end{definition}

Combining the definition with Proposition~\ref{prop:tildeEF} connects
this with the ordinary geometric fixed points.

\begin{corollary}\label{cor:GeomFP}
Given any $G$-Mackey functor $\mM$, we have a natural isomorphism of
$G/N$-Mackey functors
\[
\m{\pi}_0^{G/N} (\Phi^N H\mM)\cong\Phi^N\mM.
\]
\end{corollary}

The geometric fixed-points can be derived by cofibrantly replacing
$\mM$ or $\mM_{\bullet}$.  Notice however that we are not using a
cofibrant replacement of $\tildeEF_{N}\mA$ here; although the result
would have quasi-isomorphic box product if we did this,
Proposition~\ref{prop:pullback} would not hold.

\begin{remark}\label{rem:LeftQuillen}
The functor $\Phi^{N}$ is in fact a left Quillen functor. If
$\mP_{\bullet}$ is a dg-Mackey functor such that
$\mP_{n}=\mA^G_{T_{n}}$ for some $G$-set $T$, then
\[
\Phi^{N}(\mP)_{n}=\mA^{G/N}_{T_{n}^{N}}
\]
is again a cofibrant object.
\end{remark}

The image of the transfer is an ideal in Green functors (and in fact,
Frobenius reciprocity says that the transfer is the analogue of an
ideal for any bilinear map), and so the Mackey functor $\tildeEF_N\mA$
is canonically a Green functor. Since pullback is strong symmetric
monoidal, the following proposition is immediate.

\begin{proposition}\label{prop:GFPSym}
The functor $\Phi^{N}$ is a strong symmetric monoidal functor.
\end{proposition}

As a consequence, the geometric fixed points functor preserves
multiplicative structures.

\begin{corollary}\label{cor:GFPGreen}
The geometric fixed points functor lifts to a functor
\[
\Phi^{N}\colon\Green_{G}\to\Green_{G/N}.
\]
\end{corollary}

To understand the analogue of the cyclotomic structure on our twisted
Hochschild homology, we need to understand the interaction of the
geometric fixed points functor and the norm.

\begin{theorem}\label{thm:GeomFPNorm}
For any $H$-Mackey functor $\mM$, we have an isomorphism of
$G/N$-Mackey functors
\[
\Phi^N N_H^G \mM\cong N_{HN/N}^{G/N} \Phi^{H\cap N} \mM.
\]
\end{theorem}
\begin{proof}
The diagonal map in spectra induces a canonical weak equivalence for
any cofibrant $H$-spectrum $E$:
\[
N_{HN/N}^{G/N} \Phi^{H\cap N} E\xrightarrow{\simeq} \Phi^{N} N_H^G E.
\]
Applying $\m{\pi}_0$ gives an isomorphism.  By
Corollary~\ref{cor:GeomFP} and the definition of the norm in Mackey
functors, this gives us an isomorphism
\[
N_{HN/N}^{G/N}\Phi^{H\cap N} \mM\cong \m{\pi}_0 N_{HN/N}^{G/N}
\Phi^{H\cap N} H\mM\cong \m{\pi}_0 \Phi^N N_H^G H\mM\cong \Phi^N
N_H^G\mM,
\]
where for the inner isomorphisms, we have used that the difference
between a spectrum and its zeroth Postnikov truncation is
$0$-connected.
\end{proof}
\begin{remark}

The ``diagonal'' map in spectra used in the proof of
Theorem~\ref{thm:GeomFPNorm} has a direct Mackey version that is
significantly harder to describe.  On representables, however,
it is easy to write down an explicit description.  If $\mM=\mA^H_{T}$,
then we are looking at the map
\[
\mA^{G/N}_{F_{HN/N}(G/N,T^{H\cap N})}\to \tildeEF_N\mA^G_{F_H(G,T)}
\]
which is induced by the inclusion
\[
F_{HN/N}(G/N,T^{H\cap N})\cong F_H(G,T)^N\hookrightarrow F_H(G,T).
\]
Although transfer maps are in general not multiplicative, the failure
of this to be such is exactly contained in the image of the transfer
maps that are killed for $\tildeEF_N$.
\end{remark}

\subsection{Cyclotomic Structure}\label{ssec:Cyclotomic}

Armed with a notion of geometric fixed points for Mackey functors, we can now prove that the twisted Hochschild complex enjoys a kind of cyclotomic structure.

\begin{proposition}\label{prop:Cyclotomic}
Let $H \subset G \subset S^1$ be finite subgroups, and let $N$ be a normal subgroup of $G$. For $\mR$ a commutative Green functor for $H$, there is a natural isomorphism of simplicial Green functors
\[
\Phi^{N}\big(\HC{G}_{H}(\mR)_{\bullet}\big)\cong
\HC{G/N}_{HN/N}(\Phi^{H\cap N}\mR)_{\bullet}.
\]
\end{proposition}
\begin{proof}
We apply the geometric fixed points functor $\Phi^N$ levelwise in the
simplicial set.  Since by Proposition~\ref{prop:GFPSym} we know that
$\Phi^N$ is a strong symmetric monoidal functor,
Theorem~\ref{thm:GeomFPNorm} yields a natural isomorphism
\[
\Phi^{N}\big(\HC{G}_{H}(\mR)_k\big)\cong
(\Phi^N N_H^G \mR)^{\Box k+1}\\ \cong (N_{HN/N}^{G/N} \Phi^{H\cap N}
 \mR)^{\Box k+1}\cong 
 \HC{G/N}_{HN/N}(\Phi^{H\cap N}\mR)_k.
\]
Additionally, multiplication by a generator $g$ induces multiplication
by $g$ on the geometric fixed points, so by naturality, we see that
these isomorphisms commute with all of the simplicial structure maps.
\end{proof}

\begin{corollary}\label{AlgebraicRestriction}
If \(G\) is a finite subgroup of \(S^{1}\), \(H, N\subset G\), and \(\mR\) a commutative Green functor for \(H\), then we have a ``geometric fixed points'' map of simplicial commutative rings
\[
\HC{G}_H(\mR)(G/G)\to \HC{G/N}_{HN/N}\big(\Phi^{H\cap N}\mR\big)\big((G/N)/(G/N)\big).
\]
\end{corollary}
\begin{proof}
For any Mackey functor \(\mM\), we have a natural isomorphism of abelian groups
\[
\tildeEF_{N}\mM(G/G)\cong \Phi^{N}\mM\big((G/N)/(G/N)\big).
\]
The result follows from considering the natural map of simplicial Green functors
\[
\HC{G}_H(\mR)\to \tildeEF_{N}\HC{G}(\mR)
\]
and evaluating at \(G/G\).
\end{proof}

In this context, however, passage to fixed points is more
difficult. The restriction to $H$ of the $G$-Hochschild complex for
$\mR$ is {\emph{not}} equal to the $H$-Hochschild complex for $\mR$ or
even for the restriction to $H$ of $\mR$.  However, they are always
quasi-isomorphic.  To see this, we use the edgewise subdivision
functor applied to simplicial Green functors.  (See~\cite[\S 1]{BHM}
for a review of the properties of the edgewise subdivision.)

\begin{proposition}\label{prop:IntermediateRestrictions}
If $H\subset J\subset G$, and $\m{R}$ is a commutative Green functor for $H$, then we have a natural quasi-isomorphism
\[
i_{J}^{\ast} \big(\HC{G}_{H}(\mR)_{\bullet}\big)
\simeq
\sd_{[G:J]}\big(\HC{J}_{H}(\mR)_{\bullet}\big).
\]
\end{proposition}

\begin{proof}
The argument here is a classical one (see, for instance, \cite[Section 5]{McCarthyAdams}). We include it here for completeness. Since the restriction is a strong symmetric monoidal functor on Mackey functors, we have
\[
i_{J}^{\ast} \big(\HC{G}_{H}(\mR)_k\big)\cong{}^{g}\big(i_{J}^{\ast} N_{H}^{G}\mR\big)\Box\big(i_{J}^{\ast} N_{H}^{G}\mR\big)^{\Box k}.
\]
Theorem~\ref{thm:ResofNorm} shows that we have an isomorphism of \(G\)-objects in \(J\)-Green functors
\[
i_{J}^{\ast} N_{H}^{G}\mR\cong \big(N_{H}^{J}\mR\big)^{\Box |G/J|},
\]
where the action on the right hand side is via tensor induction. This shows that the \(k\)-simplices are given by
\[
i_{J}^{\ast} \big(\HC{G}_{H}(\mR)_k\big)\cong {}^{g}\big(N_{H}^{J}\mR\big)^{\Box |G/J|}\Box\big(N_{H}^{J}\mR^{\Box |G/J|}\big)^{\Box k}\cong \big(N_{H}^{J}\mR\big)^{\Box (|G/J|(k+1))}.
\]
The face maps besides the last are induced by the multiplications on \(N_{H}^{G}\mR\), which when restricted to \(J\), give the component-wise multiplications
\[
\big(N_{H}^{J}\mR\big)^{\Box |G/J|}\Box \big(N_{H}^{J}\mR\big)^{\Box |G/J|}\to \big(N_{H}^{J}\mR\big)^{\Box |G/J|}.
\]
This is exactly the iterated face map which arises in the edgewise subdivision. Similarly, the degeneracy maps are induced by inserting the unit for \(N_{H}^{G}\mR\), which restricts to inserting a collection of units for each of the tensor factors. For the last face map, we observe that the map ``act by \(g\) and multiply'' realizes the rotation of the tensor factors, and then the action of \(g^{|G/J|}\), a generator of \(J\), taking the last to the first. This is exactly the value of the edgewise subdivision on the last face.
\end{proof}

Using Proposition~\ref{prop:ZHomotopyOrbits}, we can also describe the
restriction to other subgroups. Since any $G$-Green functor is a
$G$-object in $G$-Green functors, the restriction to a subgroup $K$ is
automatically a $G$-object in $K$-Green functors. This describes our
restriction to other subgroups.

\begin{proposition}
Let $K\subset H\subset G \subset S^1$ and let \(h\in H\) be the generator. For $\m{R}$ a commutative Green functor for $H$, we have a natural quasi-isomorphism
\[
i_{K}^{\ast} \big(\HC{G}_{H}(\mR)_{\bullet}\big)\simeq
\big(i_{K}^{\ast}\mR\big)_{h\Z},
\]
where the $\Z$-action on $i_{K}^{\ast}\mR$ is via the quotient $\Z\to
H$ sending \(1\) to \(h\). 
\end{proposition}
\begin{proof}
Proposition~\ref{prop:IntermediateRestrictions} and naturality of the
restriction allow us to reduce to the case that $H=G$ (since
otherwise, we first restrict to $J=H$ in
Proposition~\ref{prop:IntermediateRestrictions} and then continue
restricting down). Here the result is obvious, since the two
homotopical left adjoints commute.
\end{proof}

\begin{remark}
When we restrict to subgroups that neither contain \(H\) nor are contained in \(H\), then the answer is somewhat trickier to describe. If \(\mR\) is a commutative \(H\)-Green functor, then \(i_{H\cap J}^{\ast}\mR\) has an \(H\)-action that extends the natural \(H\cap J\)-action. By functoriality, the norm
\[
N_{H\cap J}^{J} i_{H\cap J}^{\ast}\mR
\] 
has an action of both \(H\) and \(J\) and they agree on \(H\cap J\). This gives an action of \(HJ\), and the restriction of \(\HC{G}_{H}(\mR)_{\ast}\) is 
\[
\big(N_{H\cap J}^{J} i_{H\cap J}^{\ast}\mR\big)_{h\Z},
\]
via a generator of \(HJ\). We do not pursue this further here.
\end{remark}

\subsection{Algebraic approximation to TR}\label{sec:Example}

In the classical theory, to compute algebraic $K$-theory using trace methods one wants to understand fixed points of topological Hochschild homology, or TR-theory. For a ring $A$, TR$^{n+1}(A;p) := \THH(A)^{C_{p^n}}$. We can consider TR as a $C_{p^n}$-Mackey functor
\[
\m{\textup{TR}}_q(A;p):= \m{\pi}^{C_{p^n}}_q \THH(A)
\]
Recall that there are two operators on TR, the Frobenius and the Restriction, which are used classically to defined topological cyclic homology. These are operators
\[
F, R: \textup{TR}^{n+1}(A;p) \to \textup{TR}^n(A;p).
\]
There is an unfortunate clash of notation here, as the Frobenius map on TR is the restriction map $res^{C_{p^n}}_{C_{p^{n-1}}}$ in the Mackey functor $\m{\textup{TR}}$. The map $R$ on TR is not part of the Mackey functor structure. This map $R$ is defined using the cyclotomic structure on THH. The spectrum TR$(A;p)$ is then defined to be 
\[
\textup{TR}(A;p) := \holim_{\substack{\longleftarrow \\ R}} \textup{TR}^n(A;p).
\]
In this section we define an algebraic analog of $\textup{TR}(A;p)$ using twisted Hochschild homology. 

We first need an analogue of the Restriction map $R$ from the topological setting. When $H=e$, for a commutative ring $A$  the ``geometric fixed points" map of Corollary \ref{AlgebraicRestriction} gives a map
\[
\m{\HH}_e^{C_{p^n}}(A)_q(C_{p^n}/C_{p^n}) \to \m{\HH}_e^{C_{p^{n-1}}}(A)_q(C_{p^{n-1}}/C_{p^{n-1}})
\]
which serves as an algebraic analogue of the restriction map $R$ on $ \textup{TR}^{n+1}(A;p)$. We use these geometric fixed point maps to define an algebraic analogue of TR. 
\begin{definition}
For a commutative ring $A$, the algebraic TR groups of $A$ are
\[
tr_k(A;p) := \lim_{\longleftarrow} \m{\HH}_{e}^{C_{p^{n}}}(A)_k(C_{p^{n}}/C_{p^{n}}).
\]
\end{definition}

Below we include an explicit computation for $A=\mathbb F_{p}$ of the algebraic TR-groups $tr_{k}(\mathbb F_{p};p)$. Recall that $\m{\HH}^{C_{p^{n}}}_{e}(\mathbb F_{p})_k$ is the homology of the twisted cyclic nerve on $N_{e}^{C_{p^n}}\mathbb F_{p}$. This Mackey functor norm can be computed following the work of Mazur \cite{MazurArxiv}. 

\begin{proposition}
For $G=C_{p^{n}}$, 
\[
N_{e}^{G}\mathbb F_{p}(G/C_{p^{k}})=\mathbb Z/p^{k+1},
\]
all restriction maps are the canonical quotients, all transfer maps are multiplication by the index (so the obvious injections) and all norms are determined by where they send $1$. The Weyl action is trivial.
\end{proposition}

Since the norm is symmetric monoidal, we know that
\[
(N_{e}^{G}\mathbb F_{p})^{\Box (k+1)}\cong N_{e}^{G}\mathbb F_{p}.
\]

Since the Weyl action is trivial, there is no difference between the $G$-twisted Hochschild complex and the ordinary Hochschild complex in Green functors, so we conclude the following.

\begin{proposition}
For all $n\geq 1$,
\[
\m{\HH}^{C_{p^{n}}}_{e}(\mathbb F_{p})_{k}=\begin{cases}
N_{e}^{C_{p^{n}}}\mathbb F_{p} & k=0 \\
0 & \text{otherwise.}
\end{cases}
\]
\end{proposition}

Note that this agrees with the classical computation, where we use the Frobenius and Verschiebung to assemble the groups \(TR_{0}^{*}(\mathbb F_{p};p)\) for \(1\leq \ast\leq (n+1)\) into a Mackey functor. Our cyclotomic structure (Proposition~\ref{prop:Cyclotomic}) connects these via the geometric fixed points. When evaluated on a point (i.e. \(G/G\)), we have geometric fixed points maps
\[
\m{\HH}^{C_{p^{n}}}_{e}(\mathbb F_{p})_0(C_{p^{n}}/C_{p^{n}})\to \m{\HH}^{C_{p^{n}}/C_{p}}_{e}(\mathbb F_{p})_0\big((C_{p^{n}}/C_{p})/(C_{p^{n}}/C_{p})\big).
\]
The target is just 
\[
\m{\HH}^{C_{p^{n-1}}}_{e}(\mathbb F_{p})_0(C_{p^{n-1}}/C_{p^{n-1}}),
\]
and the map we get is just the canonical quotient map
\[
\mathbb Z/p^{n}=\m{\HH}^{C_{p^{n}}}_{e}(\mathbb F_{p})_0(C_{p^{n}}/C_{p^{n}})\to \m{\HH}^{C_{p^{n-1}}}_{e}(\mathbb F_{p})_0(C_{p^{n-1}}/C_{p^{n-1}})
=\mathbb Z/p^{n-1}.
\]
This gives us the purely algebraic analogue of \(TR(\mhyphen;p)\).

\begin{proposition}
The algebraic $TR$ groups of $\mathbb F_{p}$ are
\[
tr_{k}(\mathbb F_{p};p):=\lim_{\longleftarrow} \m{\HH}_{e}^{C_{p^{n}}}(\mathbb F_{p})_k(C_{p^{n}}/C_{p^{n}})=\mathbb Z_{p}.
\]
\end{proposition}

In the case of $\mathbb{F}_p$, algebraic TR is an excellent approximation for the topological theory. Indeed, the result for $tr_{k}(\mathbb F_{p};p)$ agrees with TR$_{k}(\mathbb F_{p};p)$, and is the $p$-completion of $K_{k}(\mathbb F_{p})$.

\section{Witt vectors for Green functors}\label{sec:GreenWitt}

Classically, Hesselholt and Madsen ~\cite{HesselholtMadsen} show that the fixed points of topological Hochschild homology are closely related to Witt vectors. In particular, for a commutative ring $R$
\[
\pi_0(\THH(R)^{C_n}) \cong \mathbb{W}_{\langle n \rangle}(R)
\]
where $\langle n \rangle$ denotes the truncation set of natural numbers which divide $n$. Using this result and the isomorphism in Theorem~\ref{thm:ConnectiontoTHH}, it follows that twisted Hochschild homology also has a close relationship to classical Witt vectors. 

\begin{proposition}\label{prop:witt}
If $R$ is a ring, then we have a natural isomorphism
\[
\m{\HH}_e^{C_{n}}(R)_0(C_n/C_n)\cong \mathbb W_{\langle n \rangle}(R),
\]
where $\langle n \rangle$ denotes the truncation set of natural numbers which divide $n$. 
\end{proposition}

Theorem \ref{thm:ConnectiontoTHH} and the isomorphism of Proposition~\ref{prop:witt} provide motivation for a definition of Witt vectors for Green functors. 
\begin{definition}
Let $\mR$ be a Green functor for $C_n\subset S^{1}$. The $C_{mn}$-Witt
vectors for $\mR$ are defined by
\[
\m{\mathbb W}_{C_{mn}}(\mR):=\m{\HH}_{C_n}^{C_{mn}}(\mR)_0.
\]
\end{definition}

When $m=p^{k}$ for some prime $p$, then we can also describe a version
of the ``length $(k+1)$ $p$-typical Witt vectors''.

\begin{definition}
Let $\mR$ be a Green functor for $C_{n}\subset S^{1}$. Then the
``length $(k+1)$ $p$-typical Witt vectors'' for $\mR$ are defined by
\[
\m{\mathbb W}_{k+1}(\mR):=\m{\mathbb W}_{C_{np^{k}}}(\mR).
\]
\end{definition}

With this definition, we have the following $H$-relative analog of the classical relationship between Witt vectors and fixed points of topological Hochschild homology. 
\begin{proposition}
For $H\subset G\subset S^1$ finite subgroups, and $R$ a (-1)-connected commutative ring orthogonal $H$-spectrum,
\[
\m{\pi}_0^G\THH_H(R) \cong \m{\mathbb{W}}_G(\m{\pi}_0^HR)
\]
\end{proposition}

When $\mR$ is a commutative Green functor, we can describe the Witt vectors
via a universal property. First observe that by construction, we have
an isomorphism of $C_{np^{k}}$-Green functors
\[
\m{\mathbb W}_{k+1}(\mR)\cong \big(N_{C_{n}}^{C_{np^{k}}}\mR\big)_{\Z},
\]
where $\Z$ acts on the norm via the quotient $\Z\to C_{np^{k}}$ given
by specifying a generator (this is zeroth homology of Proposition~\ref{prop:ZHomotopyOrbits}). In particular, since we are considering the
actual quotient rather than the homotopy one, we have a further
identification
\[
\m{\mathbb W}_{k+1}(\mR)\cong
\big(N_{C_{n}}^{C_{np^{k}}}\mR\big)_{C_{np^{k}}}.
\]
In general, the orbits are the left adjoint to the inclusions of some kind of ``trivial'' objects in the category. The same is true here.
\begin{definition}
Let \(\cO\) be an indexing system, and let \(\cOTamb^{G}\) be the category of \(\cO\)-Tambara functors. Let \(\cOTamb^{G}_{tr}\) denote the full subcategory of \(\cOTamb^{G}\) spanned by the \(\cO\)-Tambara functors \(\mR\) such that the Weyl action on \(\mR(G/H)\) is trivial for all \(H\). Let 
\[
i_{\ast}\colon \cOTamb^{G}_{tr}\to\cOTamb^{G}
\]
be the inclusion.
\end{definition}
The following is immediate.
\begin{proposition}
The orbits functor is left-adjoint to the functor \(i_{\ast}\).
\end{proposition}

For Tambara functors, Mazur and Hoyer showed that norm is the left adjoint to the forgetful functor. This hold more generally for
\(\cO\)-Tambara functors \cite[Theorem~6.15]{BHincomplete}, provided the \(C_{np^{k}}\)-set \(C_{np^{k}}/C_{n}\) is an admissible set for \(\cO\). There is a preferred one, arising topologically. 

Let $\cO_{n}$ be the indexing system for $G=C_{np^{k}}$ where the admissible sets for a subgroup $H$ are all of those $H$-sets fixed by $C_{n}\cap H$ (The topological operad giving rise to this is the little disks operad for the universe given by infinitely many copies of the permutation representation on $G/C_{n}$). An $\cO_{n}$-algebra has an underlying Green functor for $C_{n}$ (with no norms), and it has norms connecting any two subgroups that contain $C_{n}$. 
\begin{proposition}[{\cite[Theorem 6.5]{BHincomplete}}]
The functor \(N_{C_{n}}^{C_{np^{k}}}\) is left adjoint to the forgetful functor
\[
i_{C_{n}}^{\ast}\colon \cO_{n}\mhyphen\Tamb^{C_{np^{k}}}\to\cO_{n}\mhyphen\Tamb^{C_{n}}\cong \Green^{C_{n}}.
\]
\end{proposition}
Putting this together, we see that the Green Witt vector construction is a left adjoint.

\begin{theorem}
The length $(k+1)$ $p$-typical Green Witt vectors functor is left adjoint
to the functor which takes an $\cO_{n}$-Tambara functor with trivial
$G$-action and sends it to the underlying $C_{n}$-Green functor.
\end{theorem}

We also have a notion of Ghost coordinates for the Green Witt
vectors given by the geometric fixed points. We
give two versions, one internal to Green functors and one describing
the evaluation.

\begin{definition}
Let $\mR$ be a Green functor for $C_{n}$. Then the ghost coordinates
for $\m{\mathbb W}_{C_{nm}}(\mR)$ are the maps of Green functors, parameterized by the subgroups $H$ of $C_{nm}$,
\[
\m{\mathbb W}_{C_{nm}}(\mR)\xrightarrow{\phi_{H}} \CoInd_{H}^{C_{nm}}
\Big(\tildeEF_{H} i_{H}^{\ast}\m{\mathbb W}_{C_{nm}}(\mR)\Big)
\]
adjoint to the canonical maps
\[
i_{H}^{\ast}\m{\mathbb W}_{C_{nm}}(\mR)\xrightarrow{\hat{\phi}_{H}} 
\tildeEF_{H} i_{H}^{\ast}\m{\mathbb W}_{C_{nm}}(\mR).
\]
\end{definition}

Put more conceptually, the map \(\phi_{H}\) is simply the restriction to \(H\) followed by the \(H\)-geometric fixed points. We unpack this slightly in the \(p\)-typical case to get a more nuanced understanding of these maps, describing in more detail what the ghost coordinates look like for subgroups \(H\) between \(C_{n}\) and \(C_{np^{k}}\). In Mackey functors, induction and coinduction agree and are both the right adjoint to the forgetful functor. The map \(\phi_{H}\) is then the composite
\[
\m{\mathbb W}_{C_{np^{k}}}(\mR)\xrightarrow{\m{\res}} \CoInd_{H}^{C_{np^{k}}} i_{H}^{\ast}\m{\mathbb W}_{C_{np^{k}}}(\mR) \xrightarrow{\CoInd\hat{\phi}_{H}}  \CoInd_{H}^{C_{np^{k}}}
\Big(\tildeEF_{H} i_{H}^{\ast}\m{\mathbb W}_{C_{np^{k}}}(\mR)\Big),
\]
where the first map, the unit of the adjunction, is a kind of extension to Mackey functors of the restriction map
\[
\m{\mathbb W}_{C_{np^{k}}}(\mR)(C_{np^{k}}/C_{np^{k}})\to \m{\mathbb W}_{C_{np^{k}}}(\mR)(C_{np^{k}}/H).
\]
Now if $C_{n}\subset H\subset C_{np^{k}}$, then we
have a natural isomorphism
\[
i_{H}^{\ast}\m{\mathbb W}_{C_{np^{k}}}(\mR)\cong \m{\mathbb W}_{H}(\mR)
\]
by Proposition~\ref{prop:IntermediateRestrictions}, and hence the target of the map is the coinduced \(H\)-Green functor
\[
\tildeEF_{H} \m{\mathbb W}_{H}(\mR)=\pi^{\ast}_{H} \Phi^{H} \m{\mathbb W}_{H}(\mR).
\]
Proposition~\ref{prop:Cyclotomic} shows then that the geometric fixed points in question are just the ordinary Hochschild homology \(\HH_{0}\) of the \(C_{n}\)-geometric fixed points of \(\mR\). 

\begin{remark}
When \(C_{n}\) is the trivial group, this reproduces the foundational work of Brun \cite{Brun2} building Witt vectors via the adjoints to the forgetful functor in Tambara functors. Our Ghost coordinates also are reflected in Brun's work, and we have chosen notation which matches his. The most transparent connection is via the following dictionary: an element
\[
(r_{H})\in\prod_{(H)}R,
\]
where \(H\) ranges over the conjugacy classes of subgroups, corresponds to the element
\[
\sum_{(H)} \tr_{H}^{G}n_{e}^{H}(r)\in N_{e}^{G}(R)(G/G).
\]
This is apparent from Brun's treatment; we include it only to help orient the reader.
\end{remark}

\subsection{Tambara Functors and Teichm\"uller Lifts}\label{ssec:Teich}

Classically, the Teichm\"uller maps are (compatible) multiplicative
maps
\[
R\to W_n(R)
\]
from $R$ to the length $n$ Witt vectors.  In this section, we
construct analogous maps in the setting of the twisted Hochschild
complex for the $C_{mn}$-Witt vectors.  The existence of these maps
is closely connected to the existence of a multiplicative structure on
$\mR$; although we restrict attention to the case where $\m{R}$ is a
Tambara functor, analogues of our results hold for the incomplete
Tambara functors of~\cite{BHincomplete} (i.e., $\cO$-commutative
monoid objects in Mackey functors).

The basic building block for our work in this section is the following
theorem, which constructs precursors that are maps of dg Green
functors.

\begin{theorem}\label{thm:teich}
If $\m{R}$ is a \(G\)-Tambara functor, then for all $H\subset G\subset
G'\subsetneq S^1$, we have maps of chain complexes of Green functors
\[
\HC{G'}_{G}\big((N_H^G i_H^\ast \mR)_{W_G(H)}\big)_{\ast}\to \HC{G'}_{G}(\mR)_{\ast}.
\]
Precomposing with the canonical quotient map gives us maps of chain
complexes
\[
\HC{G'}_{H}(i_{H}^{\ast}\mR)_{\ast}=\HC{G'}_{G}(N_H^G i_H^\ast\mR)_{\ast}\to \HC{G'}_{G}(\mR)_{\ast}.
\]
\end{theorem}

\begin{proof}
Mazur and Hoyer showed that Tambara functors are the $G$-commutative
monoids in the standard $G$-symmetric monoidal structure on Mackey
functors \cite{MazurArxiv, HoyerThesis}.  In particular, this means
that we have an extension of the functor
\[
T\mapsto N^T\m{R}
\]
from the category of finite $G$-sets and isomorphisms to the full
category of finite $G$-sets and moreover that this extension is
compatible with restriction to subgroups.  Restricting attention to
$T=G/H$, we have as part of this data a map (necessarily of
commutative Green functors)
\[
N^{G/H} \m{R}\xrightarrow{n_H^G} \m{R},
\]
such that for all $g\in W_G(H)$, we have a commutative diagram
\[
\xymatrix{ {N^{G/H}\m{R}}\ar[r]^-{n_H^G}\ar[d]_-{g} & {\mR}
  \\ {N^{G/H}\m{R}}\ar[ru]_-{n_H^G}. }
\]
(Note that here we have implicitly used
Proposition~\ref{prop:LanGmodH} to understand the external action of
the Weyl group on the whole Mackey functor).  Therefore, we have
induced maps  
\[
N_H^Gi_H^\ast\mR\to\Big(N_H^G i_H^\ast\mR\Big)_{W_G(H)}\to\mR
\]
of $G$-Green functors, and the theorem follows by applying $\HC{G'}_{\ast}(\mhyphen)$
to these maps.
\end{proof}

The maps of Theorem~\ref{thm:teich} are the antecedent to the
Teichm\"uller maps on the ordinary Witt vectors, linking the values of
the dg Green functor together for different orbits $G/H$.  To apply
this, we need to recall how the classical norms on a Tambara functor
fit into the externalized version.

\begin{proposition}[{\cite{MazurArxiv, HoyerThesis}}]\label{prop:Norms}
For any $H\subset G$ and for any Tambara functor $\m{R}$, there is a
multiplicative map
\[
N_H^G\colon \m{R}(G/H)\to N^{G/H}\m{R}(G/G).
\]
If $g\in W_G(H)$, then we have a commutative square
\[
\xymatrix{ 
{\m{R}(G/H)}\ar[r]^-{N_H^G}\ar[d]_g & {N^{G/H}\m{R}(G/G)}\ar[d]^{g} \\ 
{\m{R}(G/H)}\ar[r]_-{N_H^G} & {N^{G/H}\m{R}(G/G).}  }
\]
These give the norm maps in the usual definition of a Tambara functor.
\end{proposition}

We now construct the Teichm\"uller maps on the $C_{mn}$-Witt vectors.

\begin{theorem}
Let $\m{R}$ be a \(G\)-Tambara functor for \(G=C_{nm}\) a cyclic group.  For all $H\subset G$, we have a natural Teichm\"uller-style, multiplicative map
\[
\m{R}(G/H)\to \big(\m{R}(G/H)\big)_{W_G(H)}\to \big(\m{\HH}^{G}_0(\m{R})\big)(G/G)=\m{\mathbb W}_{G}(\mR)(G/G)
\]
lifting the \(|G/H|\)th power map
\[
\m{R}(G/H)\xrightarrow{r\mapsto r^{|G/H|}} \m{R}(G/H)_{W_{G}(H)}=\m{\mathbb W}_{G}(\mR)(G/H).
\]
\end{theorem}
\begin{proof}

We apply Theorem~\ref{thm:teich} to the case \(G'=G\).
Since the external norm map factors through the Weyl-coinvariants in Green functors, we have a sequence of Green functors
\[
N^{G/H}\mR\to (N^{G/H}\mR)_{W_{G}(H)}\to \mR,
\]
and composing with the map to \(\m{\mathbb W}_{G}(\mR)\), this gives a sequence of maps
\[
N^{G/H}\mR\to (N^{G/H}\mR)_{W_{G}(H)}\to \m{\mathbb W}_{G}(\mR).
\]
Evaluating this at \(G/G\) gives a sequence of commutative rings, and Proposition~\ref{prop:Norms} says that the natural norm map
\[
\mR(G/H)\to N^{G/H}\mR(G/G)
\]
is Weyl equivariant, and hence it extends to the diagram of multiplicative monoids 
\[
\begin{tikzcd}
{\mR(G/H)}
	\ar[r]
	\ar[d, "N_{H}^{G}"']
	&
{\mR(G/H)_{W_{G}(H)}}
	\ar[d, "N_{H}^{G}"]
	& 
	\\
{N^{G/H}\mR(G/G)}
	\ar[r]
	&
{(N^{G/H}\mR)_{W_{G}(H)}(G/G)}
	\ar[r]
	&
{\m{\mathbb W}_{G}(\mR)(G/G).}
\end{tikzcd}
\]

The second claim follows from the identification of the restriction of the norm: 
for any \(r\in\mR(G/H)\), we have 
\[
i_{H}^{\ast}N_{H}^{G}(r)=\prod_{g\in W_{G}(H)}gr.
\]
In the Weyl coinvariants, these all agree.
\end{proof}

To justify calling these Teichm\"uller maps, we connect them to the
classical construction in Witt vectors.  We first need an elementary
proposition.

\begin{proposition}
Let $R$ be a commutative ring and $G$ a finite group, then the
commutative Green functor $N_e^G(R)$ has a canonical Tambara functor
structure.
\end{proposition}

\begin{proof}
For a commutative ring $R$, there is a model of the Eilenberg-MacLane
functor $HR$ that is a commutative ring spectrum. Since the norm
functor is the left adjoint to the forgetful functor on commutative
ring spectra, $N_e^G HR$ is a $G$-commutative ring orthogonal
$G$-spectrum.  Since $\m{\pi}_0$ is lax $G$-symmetric monoidal, the result
now follows.
\end{proof}

As an immediate corollary, we obtain a natural ``Teichm\"uller lift''.

\begin{corollary}
For any commutative ring $R$ and for any cyclic group \(G=C_{n}\), we
have a natural, multiplicative map
\[
R\mapsto \m{\HH}_e^G(R)_0(G/G)={\mathbb W}_{\langle n\rangle}(R).
\]
\end{corollary}

This is a Teichm\"uller lift in the sense that it provides a section of
the $G$\textsuperscript{th} ghost map from $\HH_0^G(R)\to R$.  Moreover, the $H$-ghost map
for any subgroup $H\subset G$ composed with our Teichm\"uller lift is
simply the $[G:H]$\textsuperscript{th} power map; this is the defining
property of the usual Teichm\"uller maps.

\section{Twisted Hochschild homology and the relative cyclic nerve of
  $C_n$-sets}\label{sec:CyclicNerves}

On the way to the computation of the topological cyclic homology of
the dual numbers, Hesselholt and Madsen~\cite{HesselholtMadsen}
established that if $M$ is a pointed monoid, there is a natural
equivalence of cyclotomic spectra
\[
\THH(R[M])\simeq \THH(R)\wedge N^{\cyc}(M).
\]
In this section, we prove an analogous statement in the $C_n$-relative
context, giving a similar decomposition for twisted Hochschild
homology.

Recall that a pointed monoid in based $G$-sets is simply a monoid in
the category of based $G$-sets.  An important example of pointed
$G$-monoids come from norms of non-equivariant pointed monoids.

\begin{lemma}
Let $M$ be a (non-equivariant) pointed monoid.  Then $N_{e}^{C_n} M$
is a pointed monoid in $C_n$-sets.
\end{lemma}

For a pointed monoid $M$ and a $G$-Green functor $\m{R}$, the pointed
monoid algebra $\m{R}[M]$ is defined as follows.

\begin{definition}
If $M$ is a pointed monoid in $C_n$-sets, then for any Green functor
$\m{R}$, we can build a new Green functor
\[
\m{R}[M]=\m{R}_{M}/\m{R}_{\ast}:=\left(\bigoplus_{C_n\cdot m\in M/C_n}
\m{R}_{C_n\cdot m}\right)/\m{R}_{\ast},
\]
and where the multiplication is specified by
\[
\m{R}_{C_n\cdot m}\Box \m{R}_{C_n\cdot m'}\cong \m{R}_{C_n\cdot
  m\times C_n\cdot m'}\to \m{R}_{M}/\m{R}_\ast,
\]
where the last map is simply induced by the multiplication in $M$.
\end{definition}

\begin{remark}\label{rem:bredon}
If $M$ is a pointed monoid in $C_n$-sets, then for any Green functor
$\m{R}$, have a natural isomorphism of Green functors
\[
\m{R}[M]\cong \m{R}\Box \m{A}[M].
\]
In particular, $\m{R}[M]$ is also describing the $\m{R}$-valued Bredon
chains on $M$.
\end{remark}

If $M$ is a pointed monoid in $C_n$-sets, then we can define a
relative variant of the cyclic nerve (c.f.~\cite[8.1]{normTHH}).
As in the topological setting, the relative cyclic nerve is not a
cyclic set but has the structure of the $n$-fold subdivision of a
cyclic set; it is a functor over $\Lambda_n^{\op}$~\cite[1.5]{BHM}.

\begin{definition}
Let $M$ be a pointed monoid in $C_n$-sets.  We write $N^{\cyc}_{C_n}
M$ to denote the $\Lambda_n^{\op}$-set that has $k$-simplices
\[
\left(N^{\cyc}_{C_n}M\right)_k=M^{\wedge (k+1)}=N_{C_n}^{C_{n(k+1)}} (M).
\]
Here the degeneracy maps are given by the unit, the structure maps $d_i$ for
$0 \leq i < k$ are the pairwise multiplications, and the last
structure map uses the cyclic permutation and then applies the action
of $\gamma$ before multiplying.  The additional operator $\tau_k$ is
specified by the action of $C_{n(k+1)}$ on the indexed smash product.
\end{definition}

A $\Lambda_n^{\op}$-set forgets to a simplicial $C_n$-set (with the
action in degree $k$ generated by $\tau_k^{k+1}$), and so the
geometric realization has a $C_n$-action.  Moreover, the geometric
realization of a $\Lambda_n^{\op}$-set is endowed with an $S^1$-action
that extends the $C_n$ action.

When $M$ is non-equivariant, it follows by inspection that we can
describe the relative cyclic bar construction of the norm in terms of
the edgewise subdivision.

\begin{proposition}
Let $M$ be a (non-equivariant) pointed monoid.  Then there is a
natural isomorphism of $\Lambda_n^{\op}$-sets
\[
N^{\cyc}_{C_n} N_e^{C_n}M\cong sd_n N^{\cyc} M.
\]
\end{proposition}

The following theorem describes the twisted Hochschild homology of
$\m{R}[M]$.

\begin{theorem}
Fix a cyclic group $G=C_n$.  Let $M$ be a pointed monoid in $G$-sets
and $\m{R}$ a $G$-Green functor.  Then we have a natural weak
equivalence 
\[
\HC{G}_{\ast}(\m{R})\Box \m{C}_{\ast}^{cell}(N^{\cyc}_{G} M;\mA) \to
\HC{G}_{\ast}(\m{R}[M]),
\]
where here we are regarding $N^{\cyc}_{G} M$ as a simplicial $G$-set
and \(\m{C}_{\ast}^{cell}\) denotes the Mackey extension of the Bredon
cellular chains.
\end{theorem}

\begin{proof}
The key point is that there is a natural equivalence
\[
C^{cell}_{\ast}(N^{\cyc}_{G} M; \mA) \htp \HC{G}_{\ast}(\m{A}[M]),
\]
since the cellular chains takes the cartesian product to the box
product of Mackey functors and $C^{cell}_{\ast} (M; \m{A}) \cong \m{A}[M]$.
Remark~\ref{rem:bredon} implies that the natural map 
\[
\HC{G}_{\ast}(\m{R})\Box C^{cell}_{\ast}(N^{\cyc}_{G} M;\mA) \cong
\HC{G}_{\ast}(\m{R}) \Box \HC{G}_{\ast}(\m{A}[M]) \to \HC{G}_{\ast}(\m{R}[M]),
\]
where the last map is induced by the levelwise box product, is a weak
equivalence.
\end{proof}

\section{Shukla Homology}\label{sec:Shukla}

Whereas THH is constructed as a derived functor, the classical
definition of Hochschild homology is not derived.  As a consequence,
the hypothesis of flatness of the input over the base ring is often
necessary; for example, such hypotheses are standard when considering
the identification of the $E_2$ term of the B\"okstedt spectral
sequences in terms of Hochschild homology.  It is sometimes convenient
to work with the derived analogue of Hochschild homology; here we take
a simplicial free resolution of the input ring and apply Hochschild
homology.  This construction is traditionally called Shukla homology.
For example, see~\cite{Brun} and~\cite{Morrow} for applications of
Shukla homology.

There is a natural extension of the twisted Hochschild homology of
Green functors to a twisted analogue of Shukla
homology. Unfortunately, there are complications arising from
difficulties in the homological algebra of Green functors.
Specifically, we need to drop the assumption that we are working with
commutative rings or Green functors, as the free commutative ring
objects fail to be flat over the Burnside ring.

\begin{proposition}
The symmetric powers of a projective Mackey functor need not be flat. 
\end{proposition}
\begin{proof}
Let $G=C_{2}$, and let $\mP=\mA_{C_{2}}$. Then 
\[
\Sym^{2}(\mP)=\mA_{C_{2}}\oplus\mZ^{\ast}.
\]
The module $\mZ^{\ast}$ has infinite Tor dimension.
\end{proof}

\begin{remark}
For a general finite group $G$, if our projective is of the form
$\mP=\mA_{G/H}$ for $H$ a proper subgroup, then symmetric powers
$\Sym^{n}\mP$ might not be projective for $n$ dividing $[G:H]$. This
is exactly the issue that the $G$-symmetric monoidal structure in
equivariant stable homotopy fixes: the non-projective summands are
essentially just ``missing'' the norm classes which we would see if we
computed $\m{\pi}_{0} \Sym^{n}(G/H_{+})$. 
\end{remark}

To avoid the difficulties that arise in the commutative case, we
instead simply work with associative Green functors. 

\begin{proposition}
If $\m{P}$ is a projective Mackey functor, then the free associative Green functor on $\m{P}$, the box-tensor algebra $T(\m{P})$, is a projective object in Mackey functors.
\end{proposition}
\begin{proof}
The tensor algebra in Mackey functors is simply
\[
T(\mM)=\bigoplus_{n\in\mathbb N} \mM^{\Box n}.
\]
By functoriality, if $\m{P}$ is a direct summand of $\bigoplus_i\m{A}_{S_i}$, with each \(S_i\) a finite \(G\)-set, then $T(\m{P})$ is a direct summand of $T(\bigoplus_i\m{A}_{S_i})$. It suffices therefore to show that $T(\m{A}_S)$ is projective for any finite $G$-set $S$. Since these are representable Mackey functors, we have natural isomorphisms
\[
\mA_{S}^{\Box n}\cong \mA_{n\cdot S},
\]
where $n\cdot S$ is the disjoint union of $n$ copies of $S$.
\end{proof}

\begin{corollary}
Let \(\mR\) be an associative Green functor and let \(\tilde\mR_\bullet\) be a simplicial resolution of \(\mR\) by free associative Green functors. Then the underlying simplicial Mackey functor for \(\tilde{\mR}_\bullet\) is cofibrant.
\end{corollary}

Since Corollary~\ref{cor:NormHomotopical} shows that the norm of
Mackey functors preserves weak equivalences between cofibrant objects,
this also allows us to build homotopically meaningful resolutions of
\(N_H^G\mR\).

\begin{corollary}
Let \(\mR\) and \(\mR'\) be associative Green functors for \(H\) such
that there is a weak equivalence $\mR \to \mR'$.  Then for any two
simplicial resolutions $\tilde\mR_\bullet$ and
$\tilde\mR'_\bullet$ of \(\mR\) and \(\mR'\) by free associative Green
  functors, there is a zigzag of weak equivalences 
\[
N_H^G\tilde\mR_\bullet \htp N_H^G\tilde\mR'_\bullet.
\]
\end{corollary}

\begin{remark}
Since the norm commutes with sifted colimits, we know that the zeroth homology group of \(N_H^G\tilde\mR_\bullet\) is actually \(N_H^G\mR\). We make no claim, however, that the higher homology groups vanish. Moreover, this is almost never a resolution by free associative Green functors.
\end{remark}

The cofibrancy of the underlying simplicial Mackey functor for the simplicial Green functor \(N_H^G\tilde\mR_\bullet\) is all we need to define our Shukla complex, however.

\begin{definition}
Let \(\mR\) be an associative Green functor for \(H\subset G\), a finite subgroup of \(S^1\), and let \(\tilde\mR_\bullet\) be a simplicial resolution of \(\mR\) by free associative Green functors. Then the \(G\)-Shukla complex of \(\mR\) is the complex associated to the simplicial Mackey functor
\[
\big(N_H^G\tilde\mR_\bullet\big)\boxover{\big(N_H^G\tilde\mR_\bullet\big)^e}{}^{\gamma}\big(N_H^G\tilde\mR_\bullet\big),
\]
where \(\gamma\) is the generator of \(G\).
\end{definition}

Recall that Definition~\ref{def:normhhr} describes the norm in Mackey
functors via the norm in spectra.  We can describe relative norms of
Mackey functor modules over Green functors in this way as well.  Since
\(H\mA\) is a genuine equivariant commutative ring spectrum, the
category of $H\mA$ modules has internal norms (e.g.,
see~\cite{HillHopkins, BlumbergHill} and~\cite[2.20,
  \S6]{normTHH}): 
if \(E\) is an \(i_{H}^{\ast}H\mA\)-module,  
\[
{}_{H\mA}N_{H}^{G}E:= H\mA\sm{N_{H}^{G}i_{H}^{\ast}H\mA}N_{H}^{G}E,
\]
where the map $N_H^G i_H^{\ast} H\mA \to H\mA$ is the counit of the
adjunction.  Choosing a cofibrant model for $H\mA$, this internal norm
is homotopical (see~\cite[6.11]{normTHH} for analogous arguments) in the
sense that it preserves weak equivalences between cofibrant objects.
Moreover, a straightforward extension of the work of~\cite{MazurArxiv,
  HoyerThesis} establishes a correspondence between $\pi_0$ of the
relative norm and the algebraic construction.  This structure gives
another way to interpret the Shukla homology of $\m{R}$.

\begin{proposition}
The Shukla homology groups of $\m{R}$ are isomorphic to the homotopy
groups of the derived smash product 
\[
H\m{R}\sm{H\m{R}\sm{H\m{A}} H\m{R}} H\m{R}.
\]
\end{proposition}

\bibliographystyle{plain}

\bibliography{WittGreen}

\begin{thebibliography}{10}

\bibitem{normTHH}
Vigleik Angeltveit, Andrew Blumberg, Teena Gerhardt, Michael Hill, Tyler
  Lawson, and Michael Mandell.
\newblock Topological cyclic homology via the norm.
\newblock {\em Documenta Mathematica}, 23:2101--2163, 2018.

\bibitem{ClanBarwick}
C.~Barwick, E.~Dotto, S.~Glasman, D.~Nardin, and J.~Shah.
\newblock Equivariant higher categories and equivariant higher algebra.
\newblock In progress, 2017.

\bibitem{Barwick}
Clark Barwick.
\newblock Spectral {M}ackey functors and equivariant algebraic {$K$}-theory
  ({I}).
\newblock {\em Adv. Math.}, 304:646--727, 2017.

\bibitem{BlumbergHill}
A.~J. Blumberg and M.~A. Hill.
\newblock {$G$}-symmetric monoidal categories of modules over equivariant
  commutative ring spectra.
\newblock arxiv:1511.07363, 2015.

\bibitem{BHincomplete}
Andrew Blumberg and Michael Hill.
\newblock Incomplete {T}ambara functors.
\newblock {\em Algebr. Geom. Topol.}, 18(2):723--766, 2018.

\bibitem{BMcyclo}
Andrew~J. Blumberg and Michael~A. Mandell.
\newblock The homotopy theory of cyclotomic spectra.
\newblock {\em Geom. Topol.}, 19(6):3105--3147, 2015.

\bibitem{BHM}
M.~B{\"o}kstedt, W.~C. Hsiang, and I.~Madsen.
\newblock The cyclotomic trace and algebraic {$K$}-theory of spaces.
\newblock {\em Invent. Math.}, 111(3):465--539, 1993.

\bibitem{Brun}
M.~Brun.
\newblock Witt vectors and equivariant ring spectra applied to cobordism.
\newblock {\em Proc. Lond. Math. Soc. (3)}, 94(2):351--385, 2007.

\bibitem{Stolz}
M.~Brun, B.~I. Dundas, and M.~Stolz.
\newblock Equivariant structure on smash powers.
\newblock arXiv:1604.05939, 2016.

\bibitem{Brun2}
Morten Brun.
\newblock Witt vectors and {T}ambara functors.
\newblock {\em Adv. Math.}, 193(2):233--256, 2005.

\bibitem{EKMM}
A.~D. Elmendorf, I.~Kriz, M.~A. Mandell, and J.~P. May.
\newblock {\em Rings, modules, and algebras in stable homotopy theory},
  volume~47 of {\em Mathematical Surveys and Monographs}.
\newblock American Mathematical Society, Providence, RI, 1997.
\newblock With an appendix by M. Cole.

\bibitem{GoerssSchemmerhorn}
Paul Goerss and Kristen Schemmerhorn.
\newblock Model categories and simplicial methods.
\newblock In {\em Interactions between homotopy theory and algebra}, volume 436
  of {\em Contemp. Math.}, pages 3--49. Amer. Math. Soc., Providence, RI, 2007.

\bibitem{GreenleesProjective}
J.~P.~C. Greenlees.
\newblock Some remarks on projective {M}ackey functors.
\newblock {\em J. Pure Appl. Algebra}, 81(1):17--38, 1992.

\bibitem{GuillouMay}
Bertrand Guillou and J.P. May.
\newblock Models of {$G$}-spectra as presheaves of spectra.
\newblock arxiv.org:1110.3571, 2017.

\bibitem{HesselholtWitt}
Lars Hesselholt.
\newblock Witt vectors of non-commutative rings and topological cyclic
  homology.
\newblock {\em Acta Math.}, 178(1):109--141, 1997.

\bibitem{HesselholtMadsen}
Lars Hesselholt and Ib~Madsen.
\newblock On the {$K$}-theory of finite algebras over {W}itt vectors of perfect
  fields.
\newblock {\em Topology}, 36(1):29--101, 1997.

\bibitem{HillHopkins}
M.~A. Hill and M.~J. Hopkins.
\newblock Equivariant symmetric monoidal structures.
\newblock arxiv.org: 1610.03114, 2016.

\bibitem{HHR}
M.~A. Hill, M.~J. Hopkins, and D.~C. Ravenel.
\newblock On the nonexistence of elements of {K}ervaire invariant one.
\newblock {\em Ann. of Math. (2)}, 184(1):1--262, 2016.

\bibitem{HillSlice}
Michael~A. Hill.
\newblock The equivariant slice filtration: a primer.
\newblock {\em Homology Homotopy Appl.}, 14(2):143--166, 2012.

\bibitem{Hovey}
Mark Hovey.
\newblock {\em Model categories}, volume~63 of {\em Mathematical Surveys and
  Monographs}.
\newblock American Mathematical Society, Providence, RI, 1999.

\bibitem{HoyerThesis}
Rolf Hoyer.
\newblock {\em Two topics in stable homotopy theory}.
\newblock PhD thesis, University of Chicago, 6 2014.

\bibitem{KaledinMackey}
D.~Kaledin.
\newblock Derived {M}ackey functors.
\newblock {\em Mosc. Math. J.}, 11(4):723--803, 822, 2011.

\bibitem{KaledinICM}
D.~Kaledin.
\newblock {\em Motivic Structures in Non-commutative Geometry}, pages 461--496.
\newblock 2012.

\bibitem{KaledinHW}
D.~Kaledin.
\newblock Witt vectors as a polynomial functor.
\newblock {\em Selecta Mathematica}, 2017.

\bibitem{LewisMandell}
L.~Gaunce Lewis, Jr. and Michael~A. Mandell.
\newblock Equivariant universal coefficient and {K}\"unneth spectral sequences.
\newblock {\em Proc. London Math. Soc. (3)}, 92(2):505--544, 2006.

\bibitem{MazurArxiv}
Kristen Mazur.
\newblock An equivariant tensor product on {M}ackey functors.
\newblock arxiv.org: 1508.04062, 2015.

\bibitem{McCarthyAdams}
Randy McCarthy.
\newblock On operations for {H}ochschild homology.
\newblock {\em Comm. Algebra}, 21(8):2947--2965, 1993.

\bibitem{Morrow}
Mathew Morrow.
\newblock Pro-unitality and pro-excision in algebraic {$K$}-theory and cyclic
  homology.
\newblock {\em Journal f\"ur die reine und angewandte Mathematik}, 2018(736),
  2015.

\bibitem{Quillen}
Daniel~G. Quillen.
\newblock {\em Homotopical algebra}.
\newblock Lecture Notes in Mathematics, No. 43. Springer-Verlag, Berlin-New
  York, 1967.

\bibitem{RichterChZsAb}
Birgit Richter.
\newblock Homotopy algebras and the inverse of the normalization functor.
\newblock {\em J. Pure Appl. Algebra}, 206(3):277--321, 2006.

\bibitem{SchwedeShipley}
Stefan Schwede and Brooke Shipley.
\newblock Stable model categories are categories of modules.
\newblock {\em Topology}, 42(1):103--153, 2003.

\bibitem{Shipley}
Brooke Shipley.
\newblock {$H\mathbb{Z}$}-algebra spectra are differential graded algebras.
\newblock {\em Amer. J. Math.}, 129(2):351--379, 2007.

\bibitem{StricklandTambara}
Neil Strickland.
\newblock Tambara functors.
\newblock arxiv.org: 1205.2516, 2012.

\bibitem{ThevanazWebb}
Jacques Th\'evenaz and Peter Webb.
\newblock The structure of {M}ackey functors.
\newblock {\em Trans. Amer. Math. Soc.}, 347(6):1865--1961, 1995.

\bibitem{Ullman}
John Ullman.
\newblock On the slice spectral sequence.
\newblock {\em Algebr. Geom. Topol.}, 13(3):1743--1755, 2013.

\end{thebibliography}

\end{document}